\documentclass[10pt]{article}
\usepackage{amsfonts, amsmath, amsthm, xypic, amssymb, color, multicol, graphicx, times, enumerate, setspace, fullpage, thmtools, thm-restate}

\declaretheorem[numberwithin=section]{corollary}
\declaretheorem[sibling=corollary]{lemma}
\declaretheorem[sibling=corollary]{proposition}
\declaretheorem[sibling=corollary]{theorem}

\theoremstyle{definition}
\declaretheorem[sibling=corollary]{definition}

\theoremstyle{remark}
\declaretheorem{remark}

\title{On the image of the Galois representation associated to a non-CM Hida family}
\author{Jaclyn Lang\thanks{The author is supported by NSF grant DGE-1144087.}
}
\date{}

\usepackage{color}

\newcommand{\leftexp}[2]{{\vphantom{#2}}^{#1}{#2}}

\newcommand{\Aa}{\mathfrak{a}}

\newcommand{\bb}{\mathfrak{b}}

\newcommand{\bchi}{\bar{\chi}}
\newcommand{\bdelta}{\bar{\delta}}
\newcommand{\bbeta}{\bar{\eta}}
\newcommand{\bgamma}{\bar{\gamma}}
\newcommand{\blambda}{\bar{\lambda}}
\newcommand{\bpi}{\bar{\pi}}
\newcommand{\brho}{\bar{\rho}}
\newcommand{\bsigma}{\bar{\sigma}}

\newcommand{\bvarepsilon}{\bar{\varepsilon}}
\newcommand{\bvarphi}{\bar{\varphi}}
\newcommand{\bzeta}{\bar{\zeta}}
\newcommand{\cc}{\mathfrak{c}}

\newcommand{\cC}{\mathcal{C}}

\newcommand{\E}{\mathbb{E}}
\newcommand{\F}{\mathbb{F}}

\newcommand{\G}{\mathbb{G}}
\newcommand{\cG}{\mathcal{G}}

\newcommand{\h}{\mathcal{H}}
\newcommand{\hh}{\textbf{h}}
\newcommand{\I}{\mathbb{I}}
\newcommand{\J}{\mathbb{J}}

\newcommand{\cK}{\mathcal{K}}
\newcommand{\aL}{\mathcal{L}}
\newcommand{\M}{\mathcal{M}}
\newcommand{\m}{\mathfrak{m}}
\newcommand{\n}{\mathfrak{n}}
\newcommand{\OK}{\ensuremath{\mathcal{O}}}
\newcommand{\onto}{\twoheadrightarrow}
\newcommand{\p}{\mathfrak{p}}
\newcommand{\Pp}{\mathfrak{P}}

\newcommand{\Q}{\mathbb{Q}}
\newcommand{\Qq}{\mathcal{Q}}
\newcommand{\q}{\mathfrak{q}}

\newcommand{\R}{\mathbb{R}}
\newcommand{\Ss}{\mathbb{S}}
\newcommand{\Sl}{\mathfrak{sl}}

\newcommand{\uu}{\mathfrak{u}}

\newcommand{\vv}{\mathfrak{v}}
\newcommand{\X}{\mathfrak{X}}
\newcommand{\Z}{\mathbb{Z}}

\DeclareMathOperator{\ab}{ab}
\DeclareMathOperator{\Ad}{Ad}

\DeclareMathOperator{\Aut}{Aut}

\DeclareMathOperator{\Frob}{Frob}
\DeclareMathOperator{\Gal}{Gal}
\DeclareMathOperator{\GL}{GL}
\DeclareMathOperator{\GSp}{GSp}
\DeclareMathOperator{\Hom}{Hom}
\DeclareMathOperator{\im}{Im}
\DeclareMathOperator{\Ind}{Ind}

\DeclareMathOperator{\Ob}{Ob}

\DeclareMathOperator{\ord}{ord}
\DeclareMathOperator{\PSL}{PSL}
\DeclareMathOperator{\SL}{SL}
\DeclareMathOperator{\Spec}{Spec}
\DeclareMathOperator{\tr}{tr}
\DeclareMathOperator{\univ}{univ}

\begin{document}

\maketitle

\begin{abstract} Fix\let\thefootnote\relax\footnote{UCLA Mathematics Department, Box 951555, Los Angeles, CA 90095-1555, USA, jaclynlang@math.ucla.edu \\
              \indent \,\,\, Tel.: (303) 587-4174\\
              \indent \,\,\, Fax: (310) 206-6673} a prime $p > 2$.  Let $\rho : \Gal(\overline{\Q}/\Q) \to \GL_2(\I)$ be the Galois representation coming from a non-CM irreducible component $\I$ of Hida's $p$-ordinary Hecke algebra.  Assume the residual representation $\brho$ is absolutely irreducible.  Under a minor technical condition we identify a subring $\I_0$ of $\I$ containing $\Z_p[[T]]$ such that the image of $\rho$ is large with respect to $\I_0$.  That is, $\im \rho$ contains $\ker(\SL_2(\I_0) \to \SL_2(\I_0/\Aa))$ for some non-zero $\I_0$-ideal $\Aa$.  This paper builds on recent work of Hida who showed that the image of such a Galois representation is large with respect to $\Z_p[[T]]$.  Our result is an $\I$-adic analogue of the description of the image of the Galois representation attached to a non-CM classical modular form obtained by Ribet and Momose in the 1980s.
\end{abstract}

\section{Introduction}
A Hida family $F$ that is an eigenform and coefficients in a domain $\I$ has an associated Galois representation $\rho_F : \Gal(\overline{\Q}/\Q) \to \GL_2(Q(\I))$, where $Q(\I)$ is the field of fractions of $\I$.  A fundamental problem is to understand the image of such a representation.  One expects the image to be ``large" in an appropriate sense, so long as $F$ does not have any extra symmetries; that is, as long as $F$ does not have CM.  (In the CM case there is a non-trivial character $\eta$ such that $\rho_F \cong \rho_F \otimes \eta$.  This forces the image of $\rho_F$ to be ``small".)  This notion of ``largeness" can be defined relative to any subring $\I_0$ of $\I$, and one can then ask if $\im \rho_F$ is large with respect to $\I_0$.  Even when $F$ does not have CM it might happen that there is an automorphism $\sigma$ of $\I$ and a non-trivial character $\eta$ such that $\rho_F^\sigma \cong \rho_F \otimes \eta$.  Such automorphisms, called conjugate self-twists of $F$, can be thought of as a weak symmetries of $F$.  In this paper we explain how conjugate self-twists constrict the image of $\rho_F$.  In particular, let $\I_0$ be the subring of $\I$ fixed by all conjugate self-twists of $F$.  Our main result is that $\im \rho_F$ is ``large" with respect to $\I_0$.

The study of the image of the Galois representation attached to a modular form, and showing that it is large in the absence of CM, was first carried out by Serre \cite{Serre} and Swinnerton-Dyer \cite{SD} in the early 1970s.  They studied the Galois representation attached to a modular form of level one with integral coefficients.  In the 1980s, Ribet \cite{R80}, \cite{R85} and Momose \cite{Mo} generalized the work of Serre and Swinnerton-Dyer to cover all Galois representations coming from classical modular forms.  Ribet's work dealt with the weight two case, and Momose proved the general case.  The main theorem in this paper is an analogue of their results in the $\I$-adic setting.  In fact, their work is a key input for our proof.  

Shortly after Hida constructed the representations $\rho_F$, Mazur and Wiles \cite{MW} showed that if $\I = \Z_p[[T]]$ and the image of the residual representation $\brho_F$ contains $\SL_2(\F_p)$ then $\im \rho_F$ contains $\SL_2(\Z_p[[T]])$.  Under the assumptions that $\I$ is a power series ring in one variable and the image of the residual representation $\brho_F$ contains $\SL_2(\F_p)$, our main result was proved by Fischman \cite{F}.  Fischman's work is the only previous work that considers the effect of conjugate self-twists on $\im \rho_F$.  Hida has shown \cite{H} under some technical hypotheses that if $F$ does not have CM then $\im \rho_F$ is large with respect to the ring $\Z_p[[T]]$, even when $\I \supsetneq \Z_p[[T]]$.  The methods he developed play an important role in this paper.  The local behavior of $\rho_F$ was studied by Zhao \cite{Zhao}.  He showed that $\rho_F|_{D_p}$ is indecomposable, a result that we make use of in this paper.  Finally, Hida and Tilouine have some work showing that certain $\GSp_4$-representations associated to Siegel modular forms have large image \cite{HT}. 

Our result is the first to describe the effect of conjugate self-twists on the image of $\rho_F$ without any assumptions on $\I$ and without assuming that the image of $\brho_F$ contains $\SL_2(\F_p)$.  We do need an assumption on $\brho_F$, namely that $\brho_F$ is absolutely irreducible and another small technical condition, but this is much weaker than assuming $\im \brho_F \supseteq \SL_2(\F_p)$.

\textit{Acknowledgements.}  I am grateful to my advisor, Haruzo Hida, for suggesting this problem to me and for his endless patience and insights as I worked on it.  I would like to thank Ashay Burungale for many helpful and encouraging conversations about this project.  Richard Pink and Jacques Tilouine provided me with insights into the larger context of this problem.  Finally, I am grateful for the financial support provided by UCLA and National Science Foundation that allowed me to complete this work.

\section{Main theorems and structure of paper}
We begin by fixing notation that will be in place throughout the paper.  Let $p > 2$ be prime.  Fix algebraic closures $\overline{\Q}$ of $\Q$ and $\overline{\Q}_p$ of $\Q_p$ as well as an embedding $\iota_p : \overline{\Q} \to \overline{\Q}_p$.  Let $G_\Q = \Gal(\overline{\Q}/\Q)$ be the absolute Galois group of $\Q$.  Let $\Z^+$ denote the set of positive integers.  Fix $N_0 \in \Z^+$ prime to $p$; it will serve as our tame level.  Let $N = N_0p^r$ for some fixed $r \in \Z^+$.  Fix a Dirichlet character $\chi : (\Z/N\Z)^\times \to \overline{\Q}^\times$ which will serve as our Nebentypus.  Let $\chi_1$ be the product of $\chi|_{(\Z/N_0\Z)^\times}$ with the tame $p$-part of $\chi$.  

For a valuation ring $W$ over $\Z_p$, let $\Lambda_W = W[[T]]$.  Let $\Z_p[\chi]$ be the extension of $\Z_p$ generated by the values of $\chi$.  When $W = \Z_p[\chi]$ we write $\Lambda_\chi$ for $\Lambda_W$.  When $W = \Z_p$ then we let $\Lambda = \Lambda_{\Z_p}$.  For any valuation ring $W$ over $\Z_p$, an \textit{arithmetic prime} of $\Lambda_W$ is a prime ideal of the form
\[
P_{k, \varepsilon} := (1 + T - \varepsilon(1 + p)(1 + p)^k)
\]
for an integer $k \geq 2$ and character $\varepsilon : 1 + p\Z_p \to W^\times$ of $p$-power order.  We shall write $r(\varepsilon)$ for the non-negative integer such that $p^{r(\varepsilon)}$ is the order of $\varepsilon$.  If $R$ is a finite extension of $\Lambda_W$, then we say a prime of $R$ is \textit{arithmetic} if it lies over an arithmetic prime of $\Lambda_W$.  

For a Dirichlet character $\psi : (\Z/M\Z)^\times \to \overline{\Q}^\times$, let $S_k(\Gamma_0(M), \psi)$ be the space of classical cusp forms of weight $k$, level $\Gamma_0(M)$, and Nebentypus $\psi$.  Let $h_k(\Gamma_0(M), \psi)$ be the Hecke algebra of $S_k(\Gamma_0(M), \psi)$.  Let $\omega$ be the $p$-adic Teichmuller character.  We can describe Hida's big $p$-ordinary Hecke algebra $\hh^{\ord}(N, \chi; \Lambda_\chi)$ as follows \cite{H}.  It is the unique $\Lambda_\chi$-algebra that is 
\begin{enumerate}
\item free of finite rank over $\Lambda_\chi$, 
\item equipped with Hecke operators $T(n)$ for all $n \in \Z^+$, 
\item satisfies the following specialization property: for every arithmetic prime $P_{k, \varepsilon}$ of $\Lambda_\chi$ there is an isomorphism
\[
\hh^{\ord}(N, \chi; \Lambda_\chi)/P_{k, \varepsilon}\hh^{\ord}(N, \chi; \Lambda_\chi) \cong h_k(\Gamma_0(Np^{r(\varepsilon)}), \chi_1\varepsilon\omega^{-k})
\]
\end{enumerate}
that sends $T(n)$ to $T(n)$ for all $n \in \Z^+$.

For a commutative ring $R$, we use $Q(R)$ to denote the total ring of fractions of $R$.  Hida has shown \cite{H86b} that there is a Galois representation
\[
\rho_{N_0, \chi} : G_\Q \to \GL_2(Q(\hh^{\ord}(N_0, \chi; \Lambda_\chi)))
\]
that is unramified outside $N$ and satisfies $\tr \rho_{N_0, \chi}(\Frob_\ell) = T(\ell)$ for all primes $\ell$ not dividing $N$.  Let $\Spec \I$ be an irreducible component of $\Spec \hh^{\ord}(N_0, \chi; \Lambda_\chi)$.  Assume further that $\I$ is primitive in the sense of Section 3 of \cite{H86a}.  Let $\lambda_F : \hh^{\ord}(N_0, \chi; \Lambda_\chi) \to \I$ be the natural $\Lambda_\chi$-algebra homomorphism coming from the inclusion of spectra.  By viewing $Q(\hh^{\ord}(N_0, \chi; \Lambda_\chi)) = \hh^{\ord}(N_0, \chi; \Lambda_\chi) \otimes_{\Lambda_\chi} Q(\Lambda_\chi)$ and composing $\rho_{N_0, \chi}$ with $\lambda_F \otimes 1$ we obtain a Galois representation
\[
\rho_F : G_\Q \to \GL_2(Q(\I))
\]
that is unramified outside $N$ and satisfies 
\[
\tr \rho_F(\Frob_\ell) = \lambda_F(T(\ell))
\]
for all primes $\ell$ not dividing $N$.

Henceforth for any $n \in \Z^+$ we shall let $a(n, F)$ denote $\lambda_F(T(n))$.  Let $F$ be the formal power series in $q$ given by
\[
F = \sum_{n = 1}^\infty a(n, F)q^n.
\]
Let $\I' = \Lambda_\chi[\{a(\ell, F) : \ell \nmid N\}]$ which is an order in $Q(\I)$ since $F$ is primitive.  We shall consider this Hida family $F$ and the associated ring $\I'$ to be fixed throughout this paper.  For a local ring $R$ we will use $\m_R$ to denote the unique maximal ideal of $R$.  Let $\F := \I'/\m_{\I'}$ the residue field of $\I'$.  We exclusively use the letter $\Pp$ to denote a prime of $\I$, and $\Pp'$ shall always denote $\Pp \cap \I'$.  Conversely, we exclusively use $\Pp'$ to denote a prime of $\I'$ in which case we are implicitly fixing a prime $\Pp$ of $\I$ lying over $\Pp'$. 

If $\Pp$ is a height one prime of $\I$ then we write $f_\Pp$ for the $p$-adic modular form obtained by reducing the coefficients of $F$ modulo $\Pp$.  In particular, if $\Pp$ is an arithmetic prime lying over $P_{k, \varepsilon}$ then $f_\Pp \in S_k(\Gamma_0(Np^{r(\varepsilon)}), \varepsilon\chi_1\omega^{-k})$.

Recall that Hida \cite{H86b} has shown that there is a well defined residual representation $\overline{\rho}_F : G_\Q \to \GL_2(\I/\m_\I)$ of $\rho_F$.  Throughout this paper we impose the following assumption.
\begin{equation}\tag{abs}\label{absolutely irreducible}
\text{Assume that } \brho_F \text{ is absolutely irreducible.}
\end{equation}
By the Chebotarev density theorem, we see that $\tr \brho_F$ is valued in $\F$.  Under \eqref{absolutely irreducible} we may use pseudo representations to find a $\GL_2(\I')$-valued representation that is isomorphic to $\rho_F$ over $Q(\I)$.  Thus we may (and do) assume that $\rho_F$ takes values in $\GL_2(\I')$.

\begin{definition}
Let $g = \sum_{n = 1}^\infty a(n, g)q^n$ be either a classical Hecke eigenform or a Hida family of such forms.  Let $K$ be the field generated by $\{a(n, g) : n \in \Z^+\}$ over either $\Q$ in the classical case or $Q(\Lambda_\chi)$ in the $\Lambda_\chi$-adic case.  We say a pair $(\sigma, \eta_\sigma)$ is a  \underline{conjugate self-twist} of $g$ if $\eta_\sigma$ is a Dirichlet character, $\sigma$ is an automorphism of $K$, and 
\[
\sigma(a(\ell, g)) = \eta_\sigma(\ell)a(\ell, g)
\]  
for all but finitely many primes $\ell$.  If there is a non-trivial character $\eta$ such that $(1, \eta)$ is a conjugate self-twist of $g$, then we say that $g$ has \underline{complex multiplication} or \underline{CM}.  Otherwise, $g$ does not have CM.
\end{definition}
If a modular form does not have CM then a conjugate self-twist is uniquely determined by the automorphism.

We shall always assume that our fixed Hida family $F$ does not have CM.  Let  
\[
\Gamma = \{\sigma \in \Aut(Q(\I)) : \sigma \text{ is a conjugate self-twist of } F\}.
\]
Under assumption \eqref{absolutely irreducible} it follows from a lemma of Carayol and Serre (Proposition 2.13 \cite{MFG}) that if $\sigma \in \Gamma$ then $\rho_F^\sigma \cong \rho_F \otimes \eta_\sigma$ over $\I'$.  As $\rho_F$ is unramified outside $N$ we see that in fact $\sigma(a(\ell, F)) = \eta_\sigma(\ell)a(\ell, F)$ for all primes $\ell$ not dividing $N$.  Therefore $\sigma$ restricts to an automorphism of $\I'$.  Let $\I_0 = (\I')^\Gamma$.  Define
\[
H_0 := \bigcap_{\sigma \in \Gamma} \ker \eta_\sigma 
\]
and 
\[
H := H_0 \cap \ker(\det(\overline{\rho}_F)).
\]
These open normal subgroups of $G_\Q$ play an important role in our proof.

For a commutative ring $B$ and ideal $\bb$ of $B$, write
\[
\Gamma_B(\bb) := \ker(\SL_2(B) \to \SL_2(B/\bb)).
\]
We call $\Gamma_B(\bb)$ a \textit{congruence subgroup} of $\GL_2(B)$ if $\bb \neq 0$.  We can now define what we mean when we say a representation is ``large" with respect to a ring.
\begin{definition}
Let $G$ be a group, $A$ be a commutative ring, and $r: G \to \GL_2(A)$ be a representation.  For a subring $B$ of $A$, we say that $r$ is \underline{$B$-full} if there is some $\gamma \in \GL_2(A)$ such that $\gamma(\im \rho) \gamma^{-1}$ contains a congruence subgroup of $\GL_2(B)$.
\end{definition}

Let $D_p$ be the decomposition group at $p$ in $G_{\Q_p} := \Gal(\overline{\Q}_p/\Q_p)$ under the embedding $G_{\Q_p} \hookrightarrow G_\Q$ induced by $\iota_p$.  Recall that over $Q(\I)$ the local representation $\rho_F|_{D_p}$ is isomorphic to $\bigl(\begin{smallmatrix} 
\varepsilon & u\\
0 & \delta
\end{smallmatrix}\bigr)$ (Theorem 4.3.2 \cite{GME}).  Let $\bvarepsilon$ and $\bdelta$ denote the residual characters of $\varepsilon$ and $\delta$, respectively.  

\begin{definition}
For any open subgroup $G_0 \leq G_\Q$ we say that $\overline{\rho}_F$ is \underline{$G_0$-regular} if $\bvarepsilon|_{D_p \cap G_0} \neq \bdelta|_{D_p \cap G_0}$. 
\end{definition}

The main result of this paper is the following.

\begin{theorem}\label{main result}
Assume $p > 2$ and let $F$ be a primitive non-CM $p$-adic Hida family.  Assume $|\F| \neq 3$ and that the residual representation $\overline{\rho}_F$ is absolutely irreducible and $H_0$-regular.  Then $\rho_F$ is $\I_0$-full.
\end{theorem}

The strategy of the proof is to exploit the results of Ribet \cite{R80}, \cite{R85} and Momose \cite{Mo}.    Since an arithmetic specialization of a non-CM Hida family cannot be CM, their work implies that if $\Pp'$ is an arithmetic prime of $\I'$ then there is a certain subring $\OK \subseteq \I'/\Pp'$ for which $\rho_F \bmod \Pp'$ is $\OK$-full.  To connect their ring $\OK$ with $\I_0$, in section \ref{proof of main thm} we show that $Q(\OK) = Q(\I_0/\Qq)$, where $\Qq = \I_0 \cap \Pp'$.  The proof that $Q(\OK) = Q(\I_0/\Qq)$ relies on establishing a relationship between conjugate self-twists of $F$ and conjugate self-twists of the arithmetic specializations of $F$.  As this may be of independent interest, we state the result here and give two different proofs in the paper.

\begin{theorem}\label{main lifting result}
Let $\Pp$ be an arithmetic prime of $\I$ and $\sigma$ be a conjugate self-twist of $f_\Pp$ that is also an automorphism of the local field $\Q_p(\{a(n, f_\Pp) : n \in \Z^+\})$.  Then $\sigma$ can be lifted to $\tilde{\sigma} \in \Gamma$ such that $\tilde{\sigma}(\Pp') = \Pp'$, where $\Pp' = \Pp \cap \I'$.
\end{theorem}

The first proof, in section \ref{deformations}, uses abstract deformation theory and applies in the generality we have stated above.  The second proof is relegated to the Appendix.  It uses automorphic methods and only allows us to lift $\sigma$ when $\eta_\sigma$ is quadratic and $p$ does not divide the conductor of $\eta_\sigma$.  The automorphic description has the advantage of being somewhat more concrete and intuitive than the deformation theoretic one.

The remainder of the paper consists of a series of reduction steps that allow us to deduce our theorem from the aforementioned results of Ribet and Momose.  Our methods make it convenient to modify $\rho_F$ to a related representation $\rho : H \to \SL_2(\I_0)$ and show that $\rho$ is $\I_0$-full.  We axiomatize the properties of $\rho$ at the beginning of section \ref{sufficiency open product image} and use $\rho$ in the next three sections to prove Theorem \ref{main result}.  Then in section \ref{same basis} we explain how to show the existence of $\rho$ with the desired properties. 

The task of showing that $\rho$ is $\I_0$-full is done in three steps.  In section \ref{sufficiency open product image} we consider the projection of $\im \rho$ to $\prod_{\Qq | P} \SL_2(\I_0/\Qq)$, where $P$ is an arithmetic prime of $\Lambda$ and $\Qq$ runs over all primes of $\I_0$ lying over $P$.  We show that if the image of $\im \rho$ in $\prod_{\Qq | P} \SL_2(\I_0/\Qq)$ is open, then $\rho$ is $\I_0$-full.  This uses Pink's theory of Lie algebras for $p$-profinite subgroups of $\SL_2$ over $p$-profinite semilocal rings \cite{P} and the related techniques developed by Hida \cite{H}.  

In section \ref{product image} we show that if the image of $\im \rho$ in $\SL_2(\I_0/\Qq)$ is $\I_0/\Qq$-full for all primes $\Qq$ of $\I_0$ lying over $P$, then the image of $\im \rho$ is indeed open in $\prod_{\Qq|P} \SL_2(\I_0/\Qq)$.  The argument is by contradiction and uses Goursat's Lemma.  It was inspired by an argument of Ribet \cite{R75}.  It is only in this section that we make use of the assumption that $|\F| \neq 3$.  

The final step showing that the image of $\im \rho$ in $\SL_2(\I_0/\Qq)$ is $\I_0/\Qq$-full for every $\Qq$ lying over $P$ is done in section \ref{proof of main thm}.  The key input is Theorem \ref{main lifting result} from section \ref{deformations} together with the work of Ribet and Momose on the image of the Galois representation associated to a non-CM classical modular form.  We give a brief exposition of their work and a precise statement of their result at the beginning of section \ref{proof of main thm}.  We reiterate the structure of the proof of Theorem \ref{main result} at the end of section \ref{proof of main thm}.  

\section{Lifting twists via deformation theory}\label{deformations}
Let $\Pp_1$ and $\Pp_2$ be (not necessarily distinct) arithmetic primes of $\I$, and let $\Pp_i' = \Pp_i \cap \I'$.  We shall often view $\Pp_i$ as a geometric point in $\Spec(\I)(\overline{\Q}_p)$.  Since $F$ is primitive, $f_{\Pp_i}$ is either a newform or the $p$-stabilization of a newform.  Therefore 
\begin{equation}\label{same quotient field}
Q(\I/\Pp_i) = \Q_p(\{a(n, f_{\Pp_i}) : n \in \Z^+\}) = \Q_p(\{a(\ell, f_{\Pp_i}) : \ell \nmid N\}) = Q(\I'/\Pp_i').
\end{equation}
Suppose there is an isomorphism $\sigma : \I/\Pp_1 \cong \I/\Pp_2$ and a Dirichlet character $\eta : G_\Q \to Q(\I/\Pp_2)^\times$ such that
\[
\sigma(a(\ell, f_{\Pp_1})) = \eta(\ell)a(\ell, f_{\Pp_2})
\]
for all primes $\ell$ not dividing $N$.  In this section we show that $\sigma$ can be lifted to a conjugate self-twist of $F$. 

\begin{theorem}\label{deformation lifting}
Assume that $\eta$ takes values in $\Z_p[\chi]$.  Then there is an automorphism $\tilde{\sigma} : \I' \to \I'$ such that 
\[
\tilde{\sigma}(a(\ell, F)) = \eta(\ell)a(\ell, F)
\]
for all but finitely many primes $\ell$ and $\sigma \circ \Pp_1 = \Pp_2 \circ \tilde{\sigma}$.  In particular, $\Pp_1'$ and $\Pp_2'$ necessarily lie over the same prime of $\I_0$.
\end{theorem}  

Let $W$ be the ring of Witt vectors of $\F$.  Let $\Q^N$ be the maximal subfield of $\overline{\Q}$ unramified outside $N$ and infinity, and let $G^N_\Q := \Gal(\Q^N/\Q)$.  Note that $\rho_F$ factors through $G^N_\Q$.  For the remainder of this section we shall consider $G^N_\Q$ to be the domain of $\rho_F$ and $\brho_F$.  

We use universal deformation rings in the proof of Theorem \ref{deformation lifting}.  For our purposes universal deformation rings of pseudo representations are sufficient.  However, since we are assuming that $\brho_F$ is absolutely irreducible, we use universal deformation rings of representations to avoid introducing extra notation for pseudo representations.  

We set up the necessary notation.  Let $\cC$ denote the category of complete local $p$-profinite $W$-algebras with residue field $\F$.  Let $\bpi : G_\Q^N \to \GL_n(\F)$ be an absolutely irreducible representation.  We say an object $R_{\bpi} \in \cC$ and representation $\bpi^{\univ} : G_\Q^N \to \GL_n(R_{\bpi})$ is a \textit{universal couple} for $\bpi$ if: for every $A \in \cC$ and representation $r : G_\Q^N \to \GL_n(A)$ such that $r \bmod \m_A \cong \bpi$, there exists a unique $W$-algebra homomorphism $\alpha(r) : R_{\bpi} \to A$ such that $r \cong \alpha(r) \circ \bpi^{\univ}$.  Mazur proved that a universal couple always exists (and is unique) when $\bpi$ is absolutely irreducible \cite{Ma}.

It is easy to see that $R_{\bpi}$ is canonically an $R_{\det \bpi}$-algebra.  Furthermore, if $G_p^{\ab}$ is the maximal $p$-profinite abelian quotient of $G^N_\Q$, then $R_{\det \bpi} = W[[G_p^{\ab}]]$ (Theorem 2.21 \cite{MFG}).  There is a finite group $\Delta$ such that $W[[G_p^{\ab}]]$ is noncanonically isomorphic to $\Lambda_W[\Delta]$, and we fix an isomorphism between them once and for all.  In particular, all of the universal deformation rings we consider have a $\Lambda_W$-algebra structure.

Since $\eta$ takes values in $\Z_p[\chi]$ which may not be contained in $W$, we need to extend scalars.  Let $\OK$ be the composite of $W$ and $\Z_p[\chi]$.  We recommend the reader assume $\OK = W$ on the first read.  For a commutative $W$-algebra $A$, let ${}^\OK\!A := \OK \otimes_W A$.  It will be important that we are tensoring on the left by $\OK$ as we will sometimes want to view ${}^\OK\!A$ as a right $W$-algebra.

We consider the universal couples $(R_{\brho_F}, \brho_F^{\univ})$, $(R_{\brho_F^{\bsigma}}, (\brho_F^{\bsigma})^{\univ})$, and $(R_{\bbeta \otimes \brho_F}, (\bbeta \otimes \brho_F)^{\univ})$.  There are canonical maps between all of these rings which will be described shortly.  The automorphism $\bsigma$ of $\F$ induces an automorphism $W(\bsigma)$ on $W$.  For any $W$-algebra $A$, let $A ^{\bsigma} := A \otimes_{W(\bsigma)} W$, where $W$ is considered as a $W$-algebra via $W(\bsigma)$.  Note that $A^{\bsigma}$ is a $W$-bimodule with different left and right actions.  Namely $w(a \otimes w') = aw \otimes w'$, which may be different from $(a \otimes w')w = a\otimes ww'$.  In particular, ${}^\OK\!A^{\bsigma} = \OK \otimes_W A \otimes_{W(\bsigma)} W$.  Let $\iota(\bsigma, A) : A \to A^{\bsigma}$ be the usual map given by $\iota(\bsigma, A)(a) = a \otimes 1$.  It is an isomorphism of rings with inverse given by $\iota(\bsigma^{-1}, A)$.  

We now define the relevant maps between the three deformation rings.  It turns out that $R_{\brho_F^{\bsigma}}$ is canonically isomorphic to $R_{\brho_F}^{\bsigma}$ as a $W$-algebra.  To see this, let $A \in \cC$ and $r : G_\Q^N \to \GL_2(A)$ be a deformation of $\brho_F^{\bsigma}$.  Then $\iota(\bsigma^{-1}, A) \circ r$ is a deformation of $\brho_F$.  By universality there is a unique $W$-algebra homomorphism $\alpha(\iota(\bsigma^{-1}, A) \circ r) : R_{\brho_F} \to A^{\bsigma^{-1}}$ such that $\iota(\bsigma^{-1}, A) \circ r \cong \alpha(\iota(\bsigma^{-1}, A) \circ r) \circ \brho_F^{\univ}$.  Tensoring $\alpha(\iota(\bsigma^{-1}, A) \circ r)$ with $W$ over $W(\bsigma)$ gives $\alpha(\iota(\bsigma^{-1}, A) \circ r) \otimes_{W(\bsigma)} 1 : R_{\brho_F}^{\bsigma} \to A$ such that $r \cong (\alpha(\iota(\bsigma^{-1}, A) \circ r) \otimes_{W(\bsigma)} 1) \circ \iota(\bsigma, R_{\brho_F}) \circ \brho_F^{\univ}$.  This shows that $R_{\brho_F}^{\bsigma}$ satisfies the universal property for $R_{\brho_F^{\bsigma}}$.  With notation as above, when $r = (\brho_F^{\bsigma})^{\univ}$ we set $\varphi = \alpha(\iota(\bsigma^{-1}, R_{\brho_F^{\bsigma}}) \circ (\brho_F^{\bsigma})^{\univ})$, so
\begin{equation}\label{univ prop of varphi}
(\brho_F^{\bsigma})^{\univ} \cong \varphi \circ \iota(\bsigma, R_{\brho_F}) \circ \brho_F^{\univ}.
\end{equation}

Let $i : R_{\bbeta \otimes \brho_F} \to {}^\OK\!{R_{\bbeta \otimes \brho_F}}$ be the map given by $x \mapsto 1 \otimes x$.  If $A$ is a $W$-algebra and $r : G_\Q^N \to \GL_2(A)$ is a deformation of $\brho_F$ then $\eta \otimes r : G_\Q^N \to \GL_2({}^\OK\!A)$ is a deformation of $\bbeta \otimes \brho_F$.  Then there is a unique $W$-algebra homomorphism $\alpha(\eta \otimes r) : R_{\bbeta \otimes \brho_F} \to {}^\OK\!A$ such that $\eta \otimes r \cong \alpha(\eta \otimes r) \circ (\bbeta \otimes \brho_F)^{\univ}$.  We can extend $\alpha(\eta \otimes r)$ to an $\OK$-algebra homomorphism $1 \otimes \alpha(\eta \otimes r) : {}^\OK\!{R_{\bbeta \otimes \brho_F}} \to {}^\OK\!A$ by sending $x \otimes y$ to $(x \otimes 1)\alpha(\eta \otimes r)(y)$.  In particular, $\eta \otimes r \cong (1 \otimes \alpha(\eta \otimes r)) \circ i \circ (\bbeta \otimes \brho_F)^{\univ}$.  When $r = \brho_F^{\univ}$, let $\psi$ denote $\alpha(\eta \otimes \brho_F^{\univ})$, so 
\[
\eta \otimes \brho_F^{\univ} \cong (1 \otimes \psi) \circ i \circ (\bbeta \otimes \brho_F)^{\univ}.
\] 
When $r = \rho_F$, let $\nu$ denote $\alpha(\eta \otimes \rho_F)$, so 
\[
\eta \otimes \rho_F \cong (1 \otimes \nu) \circ i \circ (\bbeta \otimes \brho_F)^{\univ}.
\]

Let $A$ be a $W$-algebra.  We would like to define a ring homomorphism $m(\bsigma, A) : A^{\bsigma} \to A$ such that $m(\bsigma, A) \circ \iota(\bsigma, A)$ is a lift of $\bsigma$.  When $A = \F$ we can do this by defining $m(\bsigma, \F)(x \otimes y) = \bsigma(x)y$.  Similarly, when $A = W$ we can define $m(\bsigma, W)(x \otimes y) = W(\bsigma)(x)y$.  If $A = W[T]$ or $W[[T]]$ then $A^{\bsigma} = W^{\bsigma}[T]$ or $W^{\bsigma}[[T]]$, and we can define $m(\bsigma, A)$ by simply applying $m(\bsigma, W)$ to the coefficients of the polynomials or power series.  However, for a general $W$-algebra $A$ it is not necessarily possible to define $m(\bsigma, A)$ or to lift $\bsigma$.  (If $A$ happens to be smooth over $W$ then it is always possible to lift $\bsigma$ to $A$.)

Fortunately, we do not need $m(\bsigma, A)$ to exist for all $W$-algebras; just for $\I'$.  Our strategy is to prove that, under the assumption that $\brho_F^{\bsigma} \cong \bbeta \otimes \brho_F$, the ring homomorphism $m(\bsigma, {}^\OK\!{R_{\brho_F}})$ exists.  We will then show that it can be descended to yield a lift of $\bsigma$ to $\I'$.

\begin{lemma}\label{universal lift of sigma bar}
If $\brho_F$ is absolutely irreducible and $\brho_F^{\bsigma} \cong \bbeta \otimes \brho_F$ then there is a ring homomorphism $m(\bsigma, {}^\OK\!{R_{\brho_F}}) : {}^\OK\!{R_{\brho_F}^{\bsigma}} \to {}^\OK\!{R_{\brho_F}}$ that is a lift of $m(\bsigma, \F)$.  In particular, $m(\bsigma, {}^\OK\!{R_{\brho_F}}) \circ \iota(\bsigma, {}^\OK\!{R_{\brho_F}})$ is a lift of $\bsigma$.
\end{lemma}

\begin{proof}
Note that since $\brho_F^{\bsigma} \cong \bbeta \otimes \brho_F$, by definition we have $R_{\brho_F^{\bsigma}} = R_{\bbeta \otimes \brho_F}$.  Let $\varphi$ and $\psi$ be the $W$-algebra homomorphisms defined above, and define $m(\bsigma, {}^\OK\!{R_{\brho_F}}) = (1 \otimes \psi) \circ (1 \otimes \varphi)$.  We will show that $1 \otimes \varphi$ induces $m(\bsigma, \F)$ and $1 \otimes \psi$ induces the identity on $\F$.  Note that $\F$ is the residue field of $\OK$ since $\bchi$ takes values in $\F$.  Therefore residually all of the tensor products with $\OK$ disappear.  Hence it suffices to show that $\varphi$ induces $m(\bsigma, \F)$ and $\psi$ acts trivially on $\F$.

By definition $\F$ is generated by $\{\overline{a(\ell, F)} : \ell \nmid N\}$.  Therefore it suffices to check that $\psi$ acts trivially on $\overline{a(\ell, F)}$ for any prime $\ell$ not dividing $N$.  But $\psi \circ (\bbeta \otimes \brho_F)^{\univ} \cong \eta \otimes \brho_F^{\univ}$.  Evaluating at $\Frob_\ell$, taking traces, and reducing to the residue field shows that $\psi$ induces the identity on $\F$.

Let $\bvarphi : \F \otimes_{\bsigma} \F \to \F$ be the residual map induced by $\varphi$.  By reducing \eqref{univ prop of varphi} to the residue field we find that $\bsigma \circ \brho_F \cong \bvarphi \circ \iota(\bsigma, \F) \circ \brho_F$.  By universality we conclude that $\bsigma = \bvarphi \circ \iota(\bsigma, \F)$.  But $\bsigma = m(\bsigma, \F) \circ \iota(\bsigma, \F)$ and hence $\bvarphi = m(\bsigma, \F)$, as desired.    
\end{proof}

When we write ${}^\OK\!{A^{\bsigma}}$ we are viewing $A^{\bsigma}$ as a $W$-algebra via the \textit{left} action of $W$.  That is, $w(a \otimes w') = aw \otimes w'$.  Since $\OK \subset \I'$ we have a natural multiplication ring homomorphism $m : {}^\OK\!\I' \to \I'$ given by $m(b \otimes a) = ba$ for $b \in \OK, a \in \I'$.  This induces a multiplication ring homomorphism $m \otimes 1 : {}^\OK\!{\I'^{\bsigma}} \to \I'^{\bsigma}$.  Let $\alpha = m \circ (1 \otimes \alpha(\rho_F))$ and $\beta = m \circ (1 \otimes \nu)$.  We have the following diagram in which everything commutes when it makes sense.    

 \[\xymatrix @=45pt@R=20pt{
	{}^\OK\!{R_{\brho_F}}\ar@{->}[r]^{1 \otimes \iota(\bsigma, R_{\brho_F})}\ar@{->}[d]_{\alpha} & {}^\OK\!{R_{\brho_F}^{\bsigma}}\ar@{->}[r]^{1 \otimes \varphi}\ar@{->}[d]_{\alpha \otimes 1} & {}^\OK\!{R_{\brho_F^{\bsigma}}} = {}^\OK\!{R_{\bbeta \otimes \brho_F}}\ar@{->}[d]_{\beta}\ar@{->}[r]^{1 \otimes \psi} & {}^\OK\!{R_{\brho_F}}\ar@{->}[d]_{\alpha} \\
	\I'\ar@{->}[r]^{\iota(\bsigma, \I')} & \I'^{\bsigma} & \I'\ar@{=}[r] & \I' \\
	\Lambda_W\ar@{->}[r]_{\iota(\bsigma, \Lambda_W)} \ar@{-}[u] & \Lambda_W^{\bsigma}\ar@{->}[r]_{m(\bsigma, \Lambda_W)}\ar@{-}[u] & \Lambda_W\ar@{=}[r]\ar@{-}[u] & \Lambda_W\ar@{-}[u]
}
\]  

We claim that $\alpha$ is surjective.  To see this it suffices to show that ${}^\OK\!{R_{\brho_F}}$ is generated over $\Lambda_\OK$ by 
\[
S = \{\tr \brho_F^{\univ}(\Frob_\ell) : \ell \nmid N\}.
\]
By the Chebotarev density theorem the $\Lambda_\OK$-algebra $R_{\brho_F}'$ generated by $S$ is just the universal deformation ring for the pseudo representation $\tr \brho_F$.  As $\brho_F$ is absolutely irreducible $R_{\brho_F}' = {}^\OK\!{R_{\brho_F}}$.  Since $\alpha(\tr \brho_F^{\univ}(\Frob_\ell)) = a(\ell, F)$ for all primes $\ell$ not dividing $N$ it follows that $\alpha$ is surjective.

Define $\Sigma := (1 \otimes \psi) \circ (1 \otimes \varphi) \circ (1 \otimes \iota(\bsigma, R_{\brho_F})) = m(\bsigma, {}^\OK\!{R_{\brho_F}}) \circ \iota(\bsigma, {}^\OK\!{R_{\brho_F}})$.  Since $\alpha$ is surjective, if $\Sigma(\ker \alpha) = \ker \alpha$ then $\Sigma$ descends to an automorphism of $\I'$.  Recall that $\cap_{\Pp'} \Pp' = 0$, where the intersection is taken over all arithmetic primes of $\I'$ which lie over arithmetic primes of $\Lambda_W$.  Thus $\ker \alpha = \cap_{\Pp'} \alpha^{-1}(\Pp')$.  Note that as $W$ is unramified over $\Z_p$ it does not contain any non-trivial $p$-power roots of unity.  Therefore all arithmetic primes of $W$ are of the form $P_{k, 1}$ with $k \geq 2$.  Therefore it suffices to show that $\Sigma$ acts on the set $\{\alpha^{-1}(\Pp') : \Pp' | P_{k, 1} \text{ for some } k \geq 2\}$.

Let $\Pp'$ be an arithmetic prime of $\I'$ that lies over $P_{k, 1}$.  By the commutativity of the above diagram and the definition of $m(\bsigma, \Lambda_W)$, we see that $\Sigma(1 + T - (1 + p)^k) = 1 + T - (1 + p)^k$.  In particular, $\Sigma(P_{k, 1})$ is an arithmetic prime of $\Lambda_W$.  Therefore $\alpha \circ \Sigma(\alpha^{-1}(\Pp'))$ is an arithmetic prime of $\I'$ lying over $P_{k, 1}$.  It follows that there is an automorphism $\tilde{\sigma}$ of $\I'$ such that 
\begin{equation}\label{descent of universal lift of bar sigma}
\tilde{\sigma} \circ \alpha = \alpha \circ \Sigma.
\end{equation}
We now show that $\tilde{\sigma}$ has the properties in the statement of Theorem \ref{deformation lifting}.

\begin{proof}[Proof of Theorem \ref{deformation lifting}]
By applying \eqref{descent of universal lift of bar sigma} to $\tr \brho_F^{\univ}(\Frob_\ell)$ for any prime $\ell$ not dividing $N$ and using the definition of the maps making up $\Sigma$ we see that $\tilde{\sigma}(a(\ell, F)) = \eta(\ell)a(\ell, F)$.  

It remains to show that $\sigma \circ \Pp_1 = \Pp_2 \circ \tilde{\sigma}$.  Let $j : R_{\brho_F} \to {}^\OK\!{R_{\brho_F}}$ be the usual inclusion given by $x \mapsto 1 \otimes x$.  Note that $\rho_{f_{\Pp_1}} \cong \Pp_1 \circ \alpha \circ j \circ  \brho_F^{\univ}$ and thus $\rho_{f_{\Pp_1}}^\sigma \cong \sigma \circ \Pp_1 \circ \alpha \circ j \circ  \brho_F^{\univ}$.  On the other hand, since $\rho_{f_{\Pp_1}}^\sigma \cong \eta \otimes \rho_{f_{\Pp_2}} \cong \Pp_2 \circ \beta \circ i \circ (\bbeta \otimes \brho_F)^{\univ}$ it follows from \eqref{univ prop of varphi} and the commutativity of the big diagram that 
\[
\rho_{f_{\Pp_1}}^\sigma \cong \Pp_2 \circ \alpha \circ \Sigma \circ j \circ \brho_F^{\univ}.
\]
Using \eqref{descent of universal lift of bar sigma}, by universality we conclude that $\sigma \circ \Pp_1 \circ \alpha \circ j = \Pp_2 \circ \tilde{\sigma} \circ \alpha \circ j$.  We claim that $\alpha \circ j$ surjects onto $\I'$, from which it follows that $\sigma \circ \Pp_1 = \Pp_2 \circ \tilde{\sigma}$.  Indeed, $\im \alpha \circ j = \Lambda_W[\{a(\ell, F) : \ell \nmid N\}]$ since $R_{\brho_F}$ is generated by $\{\tr \brho_F^{\univ}(\Frob_\ell) : \ell \nmid N\}$ over $\Lambda_W$.  So we just need to show that the values of $\chi$ are generated by $\{a(\ell, F) : \ell \nmid N\}$ over $\Lambda_W$.  Define $\kappa : 1 + p\Z_p \to \Lambda^\times$ by $\kappa((1 + p)^s) = (1 + T)^s$ for $s \in \Z_p$.  Recall that for $\ell \nmid N$ we have $\det \rho_F(\Frob_\ell) = \chi(\ell)\kappa(\langle \ell \rangle)\ell^{-1}$.  As $\kappa(\langle \ell \rangle)\ell^{-1} \in \Lambda^\times$ it follows that the values of $\chi$ are in $\im \alpha \circ j$.  Thus $\im \alpha \circ j = \I'$, as desired. 
\end{proof}   

We finish this section by recalling a lemma of Momose that shows that $\eta$ automatically takes values in $\Z_p[\chi]$ if $\Pp_1 = \Pp_2$.  Thus Theorem \ref{deformation lifting} says that whenever a conjugate self-twist of a classical specialization $f_\Pp$ of $F$ induces an automorphism of $\Q_p(f_\Pp)$, that conjugate self-twist can be lifted to a conjugate self-twist of the whole family $F$.
\begin{lemma}[Lemma 1.5 \cite{Mo}]\label{twist characters and Nebentypus}
If $\sigma$ is a conjugate self-twist of $f \in S_k(\Gamma_0(N), \chi)$, then $\eta_\sigma$ is the product of a quadratic character with some power of $\chi$.  In particular, $\eta_\sigma$ takes values in $\Z[\chi]$.
\end{lemma}
The proof of Lemma \ref{twist characters and Nebentypus} is not difficult and goes through without change in the $\I$-adic setting.  For completeness, we give the proof in that setting.

\begin{lemma}\label{powers of chi}
If $\sigma$ is a conjugate self-twist of $F$ then $\eta_\sigma$ is the product of a quadratic character with some power of $\chi$.  In particular, $\eta_\sigma$ has values in $\Z[\chi]$.
\end{lemma}

\begin{proof}
As $\brho_F$ is absolutely irreducible, $\rho_F^\sigma \cong \eta_\sigma \otimes \rho_F$.  Thus $\sigma(\det \rho_F) = \eta_\sigma^2 \det \rho_F$.  Recall that for all primes $\ell$ not dividing $N$ we have
\[
\det \rho_F(\ell) = \chi(\ell)\kappa(\langle \ell \rangle) \ell^{-1},
\]
where $\kappa : 1 + p\Z_p \to \Lambda^\times$ is as in the proof of Theorem \ref{deformation lifting}.  Substituting this expression for $\det \rho_F$ into $\sigma(\det \rho_F) = \eta_\sigma^2\det \rho_F$ yields $\eta_\sigma^2 = \chi^{\sigma}\chi^{-1}$.  

Recall that $\chi^{\sigma} = \chi^\alpha$ for some integer $\alpha > 0$.  To prove the result it suffices to show that there is some $i \in \Z$ such that $\eta_\sigma^2 = \chi^{2i}$.  If $\chi$ has odd order then there is a positive integer $j$ for which $\chi = \chi^{2j}$.  Thus $\eta_\sigma^2 = \chi^{\sigma - 1} = \chi^{2j(\alpha - 1)}$.  If $\chi$ has even order then $\chi^{\sigma}$ also has even order since $\sigma$ is an automorphism.  Thus $\alpha$ must be odd.  Then $\alpha - 1$ is even and $\eta_\sigma^2 = \chi^\sigma\chi^{-1} = \chi^{\alpha - 1}$, as desired.
\end{proof}  

\section{Sufficiency of open image in product}\label{sufficiency open product image}
Recall that $H_0 = \cap_{\sigma \in \Gamma} \ker(\eta_\sigma)$ and $H = H_0 \cap \ker(\det \brho_F)$.  For a variety of reasons, our methods work best for representations valued in $\SL_2(\I_0)$ rather than $\GL_2(\I')$.  Therefore, for the next three sections we assume the following theorem, the proof of which is given in section \ref{same basis}.

\begin{restatable}{theorem}{makingrho}\label{making rho}
Assume that $\brho_F$ is absolutely irreducible and $H_0$-regular.  If $V = \I'^2$ is the module on which $G_\Q$ acts via $\rho_F$, then there is a basis for $V$ such that all of the following happen simultaneously:
\begin{enumerate}
\item $\rho_F$ is valued in $GL_2(\I')$;
\item $\rho_F|_{D_p}$ is upper triangular;
\item $\rho_F|_{H_0}$ is valued in $\GL_2(\I_0)$;
\item There is a matrix $\textbf{j} = \bigl(\begin{smallmatrix} 
\zeta & 0\\
0 & \zeta'
\end{smallmatrix}\bigr)$, where $\zeta$ and $\zeta'$ are roots of unity, such that $\textbf{j}$ normalizes the image of $\rho_F$ and $\zeta \not\equiv \zeta' \bmod p$.
\end{enumerate}
\end{restatable}

Let $H' = \ker(\det \brho_F)$.  For any $h \in H'$ we have $\det \rho_F(h) \in 1 + \m_{\I'}$.  Since $p \neq 2$ and $\I'$ is $p$-adically complete, we have
\[
\sqrt{\det \rho_F(h)} = \sum_{n = 0}^\infty {{1/2}\choose{n}}(\det \rho_F(h) - 1)^n \in \I'^\times.
\]
Sine $\rho_F$ is a $2$-dimensional representation $\rho_F|_{H'} \otimes \sqrt{\det \rho_F|_{H'}}^{-1}$ takes values in $\SL_2(\I')$.  Restricting further it follows from Theorem \ref{making rho} that 
\[
\rho := \rho_F|_H \otimes \sqrt{\det \rho_F|_H}^{-1}
\]
takes values in $\SL_2(\I_0)$.  Note that the image of $\rho$ is still normalized by the matrix $\textbf{j}$ of Theorem \ref{making rho} since we only modified $\rho_F$ by scalars, which commute with $\textbf{j}$.  In the following proposition we see that $\rho_F$ is $\I_0$-full if and only if $\rho$ is $\I_0$-full.  In the next three sections we prove that $\rho$ is $\I_0$-full.

\begin{proposition}\label{transferring congruence subgroup}
Assume $|\F| \neq 3$.  The representation $\rho_F$ is $\I_0$-full if and only if $\rho$ is $\I_0$-full.
\end{proposition}

\begin{proof}
By Corollary 1 in \cite{T} we see that, so long as $|\F| \neq 3$, a subgroup $G$ of $\SL_2(\I_0)$ contains a congruence subgroup for $\I_0$ if and only if $G$ is a subnormal subgroup of $\SL_2(\I_0)$.  Note that by definition of $\rho$ we have $\im \rho_F|_H \cap \SL_2(\I_0) \subseteq \im \rho$.  Thus if $\rho_F$ is $\I_0$-full it follows immediately that $\rho$ is $\I_0$-full.  It is the converse implication that is interesting.

Assume $\rho$ is $\I_0$-full, so by Corollary 1 in \cite{T} we see that $\im \rho$ is a subnormal subgroup of $\SL_2(\I_0)$.  Let $G = \im \rho_F|_H \cap \SL_2(\I_0)$.  To see that $\rho_F$ is $\I_0$-full it suffices to show that $G$ is a subnormal subgroup of $\SL_2(\I_0)$.  Since $\im \rho$ is subnormal and $G \subseteq \im \rho$ it suffices to show that $G$ is normal in $\im \rho$.  This follows easily from the definition of $\rho$.  
\end{proof}

The purpose of the current section is to make the following reduction step in the proof of Theorem \ref{main result}.

\begin{proposition}\label{open image in product implies main result}
Assume there is an arithmetic prime $P$ of $\Lambda$ such that the image of $\im \rho$ in $\prod_{\Qq | P} \SL_2(\I_0/\Qq)$ is open in the product topology.  Then $\rho$ (and hence $\rho_F$) is $\I_0$-full.
\end{proposition}
  
In the proof we use a result of Pink \cite{P} that classifies $p$-profinite subgroups of $\SL_2(A)$ for a complete semilocal $p$-profinite ring $A$.  (Our assumption that $p > 2$ is necessary for Pink's theory.)  We give a brief exposition of the relevant parts of his work for the sake of establishing notation.   Define 
\begin{align*}
\Theta : \SL_2(A) &\to \Sl_2(A)\\
\textbf{x} &\mapsto \textbf{x} - \frac{1}{2}\tr(\textbf{x}),
\end{align*}
where we consider $\frac{1}{2}\tr(\textbf{x})$ as a scalar matrix.  Let $\cG$ be a $p$-profinite subgroup of $\SL_2(A)$.  Define $L_1(\cG)$ to be the closed subgroup of $\Sl_2(A)$ that is topologically generated by $\im \Theta$.  Let $L_1 \cdot L_1$ be the closed (additive) subgroup of $M_2(A)$ topologically generated by $\{\textbf{xy} : \textbf{x}, \textbf{y} \in \cG\}$.  Let $C$ denote $\tr(L_1 \cdot L_1)$.  Sometimes we will view $C \subset M_2(A)$ as a set of scalar matrices.  For $n \geq 2$ define $L_n(\cG)$ to be the closed (additive) subgroup of $\Sl_2(A)$ generated by 
\[
[L_1(\cG), L_{n - 1}(\cG)] := \{\textbf{xy} - \textbf{yx} : \textbf{x} \in L_1(\cG), \textbf{y} \in L_{n - 1}(\cG)\}.
\]  
\begin{definition}
The \underline{Pink-Lie algebra} of a $p$-profinite group $\cG$ is $L_2(\cG)$.  Whenever we write $L(\cG)$ without a subscript we shall always mean $L_2(\cG)$.  
\end{definition}
As an example one can compute that for an ideal $\Aa$ of $A$, the $p$-profinite subgroup $\cG = \Gamma_A(\Aa)$ has Pink-Lie algebra $L_2(\cG) = \Aa^2\Sl_2(A)$.  This example plays an important role in what follows.  

For $n \geq 1$, define
\begin{align*}
\M_n(\cG) &= C \oplus L_n(\cG) \subset M_2(A)\\
\h_n(\cG) &= \{\textbf{x} \in \SL_2(A) : \Theta(\textbf{x}) \in L_n(\cG) \text{ and } \tr(\textbf{x}) - 2 \in C\}.
\end{align*}
Pink proves that $\M_n(\cG)$ is a closed $\Z_p$-Lie algebra of $M_2(A)$ and $\h_n = \SL_2(A) \cap (1 + \M_n)$ for all $n \geq 1$.  Furthermore, write 
\[
\cG_1 = \cG, \cG_{n + 1} = (\cG, \cG_n),
\]
where $(\cG, \cG_n)$ is the closed subgroup of $\cG$ topologically generated by the commutators $\{gg_ng^{-1}g_n^{-1} : g \in \cG, g_n \in \cG_n\}$.
Pink proves the following theorem.

\begin{theorem}[Pink \cite{P}]
With notation as above, $\cG$ is a closed normal subgroup of $\h_1(\cG)$.  Furthermore, $\h_n(\cG) = (\cG, \cG_n)$ for $n \geq 2$. 
\end{theorem}

There are two important functoriality properties of the correspondence $\cG \mapsto L(\cG)$ that we will use.  First, since $\Theta$ is constant on conjugacy classes of $\cG$ it follows that $L_n(\cG)$ is stable under the adjoint action of the normalizer $N_{\SL_2(A)}(\cG)$ of $\cG$ in $\SL_2(A)$.  That is, for $\textbf{g} \in N_{\SL_2(A)}(\cG), \textbf{x} \in L_n(\cG)$ we have $\textbf{gxg}^{-1} \in L_n(\cG)$.  If $\Aa$ is an ideal of $A$ such that $A/\Aa$ is $p$-profinite, then we write $\overline{\cG}_\Aa$ for the $p$-profinite group $\cG\cdot \Gamma_A(\Aa)/\Gamma_A(\Aa) \subseteq \SL_2(A/\Aa)$.  The second functoriality property is that the canonical linear map $L(\cG) \to L(\overline{\cG}_\Aa)$ induced by $\textbf{x} \mapsto \textbf{x} \bmod \Aa$ is surjective.    

Let $\m_0$ be the maximal ideal of $\I_0$, and let $\G$ denote the $p$-profinite group $\im \rho \cap \Gamma_{\I_0}(\m_0)$.  The proof of Proposition \ref{open image in product implies main result} consists of showing that if $\overline{\G}_{P\I_0}$ is open in $\prod_{\Qq | P} \SL_2(\I_0/\Qq)$ then $\G$ contains $\Gamma_{\I_0} (\Aa_0)$ for some nonzero $\I_0$-ideal $\Aa_0$.  Let $L = L(\G)$ be the Pink-Lie algebra of $\G$.  Since $\overline{\G}_{P\I_0}$ is open, for every prime $\Qq$ of $\I_0$ lying over $P$ there is a nonzero $\I_0/\Qq$-ideal $\overline{\Aa}_{\Qq}$ such that 
\[
\overline{\G}_{P\I_0} \supseteq \prod_{\Qq | P} \Gamma_{\I_0/\Qq}(\overline{\Aa}_{\Qq}).
\]
Thus $L(\overline{\G}_{P\I_0}) \supseteq \oplus_{\Qq | P} \overline{\Aa}_{\Qq}^2\Sl_2(\I_0/\Qq)$.

Recall from Theorem \ref{making rho} that we have roots of unity $\zeta$ and $\zeta'$ such that $\zeta \not\equiv \zeta' \bmod p$ and the matrix $\textbf{j} := \bigl(\begin{smallmatrix}
\zeta & 0\\
0 & \zeta' \end{smallmatrix}\bigr)$ normalizes $\G$.  Let $\alpha = \zeta\zeta'^{-1}$.  A straightforward calculation shows that the eigenvalues of $\Ad(\textbf{j})$ acting on $\Sl_2(\I_0)$ are $\alpha, 1, \alpha^{-1}$.  Note that since $\zeta \neq \zeta'$ either all of $\alpha, 1, \alpha^{-1}$ are distinct or else $\alpha = -1$.  For $\lambda \in \{\alpha, 1, \alpha^{-1}\}$ let $L[\lambda]$ be the $\lambda$-eigenspace of $\Ad(\textbf{j})$ acting on $L$.  One computes that $L[1]$ is the set of diagonal matrices in $L$.  If $\alpha = -1$ then $L[-1]$ is the set of antidiagonal matrices in $L$.  If $\alpha \neq -1$ then $L[\alpha]$ is the set of upper nilpotent matrices in $L$, and $L[\alpha^{-1}]$ is the set of lower nilpotent matrices in $L$.  Regardless of the value of $\alpha$, let $\uu$ denote the set of upper nilpotent matrices in $L$ and $\uu^t$ denote the set of lower nilpotent matrices in $L$.  Let $\aL$ be the $\Z_p$-Lie algebra generated by $\uu$ and $\uu^t$ in $\Sl_2(\I_0)$.  

\begin{lemma}\label{Lambda module}
With notation as above, $\aL$ is a $\Lambda$-submodule of $\Sl_2(\I_0)$.
\end{lemma}

\begin{proof}
Since $\aL$ is a $\Z_p$-Lie algebra and $\Lambda = \Z_p[[T]]$, it suffices to show that $\textbf{x} \in \aL$ implies $T\textbf{x} \in \aL$.  Recall that $\textbf{J} := \bigl(\begin{smallmatrix}
1 + T & 0\\
0 & 1
\end{smallmatrix}\bigr) \in \im \rho_F$.  Since ${\rho_F}|_H$ and $\rho$ differ only by a scalar, their images have the same normalizer.  Thus $\G$ (and hence $L$) is normalized by $\textbf{J}$.  If $\textbf{x} \in \uu$ then a simple computation shows that $\textbf{JxJ}^{-1} = (1 + T)\textbf{x}$.  As $L$ is an abelian group it follows that $T\textbf{x} = (1 + T)\textbf{x} - \textbf{x} \in \uu$.  Similarly, for $\textbf{y} \in \uu^t$ we have $T\textbf{y} \in \uu^t$.  It follows that $T[\textbf{x}, \textbf{y}] = [T\textbf{x}, \textbf{y}] \in \aL$.  Any element in $\aL$ can be written as a sum of elements in $\uu, \uu^t,$ and $[\uu, \uu^t]$.  Therefore $\aL$ is a $\Lambda$-submodule of $\Sl_2(\I_0)$.
\end{proof}

The proof of Proposition \ref{open image in product implies main result} depends on whether or not $\alpha = -1$; it is easier when $\alpha \neq -1$.  

\begin{proof}[Proof of Proposition \ref{open image in product implies main result} when $\alpha \neq -1$]
We will show that the finitely generated $\Lambda$-module 
\[
X := \Sl_2(\I_0)/\aL
\]
is a torsion $\Lambda$-module.  From this it follows that there is a nonzero $\Lambda$-ideal $\Aa$ such that $\Aa\Sl_2(\I_0) \subseteq \aL$.  Thus
\[
(\Aa\I_0)^2\Sl_2(\I_0) \subseteq \aL \subseteq L
\]
since $\I_0\Sl_2(\I_0) = \Sl_2(\I_0)$.  But $(\Aa\I_0)^2\Sl_2(\I_0)$ is the Pink-Lie algebra of $\Gamma_{\I_0}(\Aa\I_0)$ and so $\Gamma_{\I_0}(\Aa\I_0) \subseteq \G_2 \subseteq \G$, as desired.

To show that $X$ is a finitely generated $\Lambda$-module, recall that the arithmetic prime $P$ in the statement of Proposition \ref{open image in product implies main result} is a height one prime of $\Lambda$.  By Nakayama's Lemma it suffices to show that $X/PX$ is $\Lambda/P$-torsion.  The natural epimorphism $\Sl_2(\I_0)/P\Sl_2(\I_0) \onto X/PX$ has kernel $\aL \cdot P\Sl_2(\I_0)/P\Sl_2(\I_0)$, so 
\[
X/PX \cong \Sl_2(\I_0/P\I_0)/(\aL \cdot P\Sl_2(\I_0)/P\Sl_2(\I_0)).
\]

We use the following notation:
\begin{align*} 
\overline{L} = L(\overline{\G}_{P\I_0}) :& \text{ the Pink-Lie algebra of } \overline{\G}_{P\I_0}\\   
\overline{L}[\lambda] :& \text{ the } \lambda\text{-eigenspace of } \Ad(\textbf{j}) \text{ on } \overline{L}, \text{ for } \lambda \in \{\alpha, 1, \alpha^{-1}\}\\
\overline{\aL} :& \text{ the } \Z_p\text{-algebra generated by } \overline{L}[\alpha] \text{ and } \overline{L}[\alpha^{-1}]
\end{align*}
The functoriality of Pink's construction implies that the canonical surjection $\I_0 \onto \I_0/P\I_0$ induces surjections
\[
L[\lambda] \onto \overline{L}[\lambda] 
\]
for all $\lambda \in \{\alpha, 1, \alpha^{-1}\}$.  Therefore the canonical linear map $\aL \to \overline{\aL}$ is also a surjection.  That is, $\aL \cdot P\Sl_2(\I_0)/P\Sl_2(\I_0) = \overline{\aL}$ and so $X/PX \cong \Sl_2(\I_0/P\I_0)/\overline{\aL}$.  Since $\overline{\G}_{P\I_0} \supseteq \prod_{\Qq | P} \Gamma_{\I_0/\Qq}(\overline{\Aa}_{\Qq})$, it follows that
\begin{align*}
\overline{L}[\alpha] &\supseteq \left\{\begin{pmatrix}
0 & x\\
0 & 0
\end{pmatrix} | x \in \oplus_{\Qq | P} \overline{\Aa}_{\Qq}^2 \right\}\\
\overline{L}[\alpha^{-1}] &\supseteq \left\{\begin{pmatrix}
0 & 0\\
x & 0
\end{pmatrix} | x \in \oplus_{\Qq | P} \overline{\Aa}_{\Qq}^2 \right\}.
\end{align*}  
Since $\alpha \neq -1$ we have $\uu = \overline{L}[\alpha]$ and $\uu^t = \overline{L}[\alpha^{-1}]$.  Therefore
\[
\overline{\aL} \supseteq \oplus_{\Qq | P} \overline{\Aa}_{\Qq}^4\Sl_2(\I_0/\Qq).
\]
Since each $\overline{\Aa}_{\Qq}$ is a nonzero $\I_0/\Qq$-ideal, it follows that $\oplus_{\Qq | P} \Sl_2(\I_0/\Qq)/\overline{\Aa}_{\Qq}^4\Sl_2(\I_0/\Qq)$ is $\Lambda/P$-torsion.  Finally, the inclusions
\[
\oplus_{\Qq | P} \overline{\Aa}^4_{\Qq}\Sl_2(\I_0/\Qq) \subseteq \overline{\aL} \subseteq \Sl_2(\I_0/P\I_0) \subseteq \oplus_{\Qq | P} \Sl_2(\I_0/\Qq\I)
\]
show that $\Sl_2(\I_0/P\I_0)/\overline{\aL} \cong X/PX$ is $\Lambda/P$-torsion.  
\end{proof}

Let
\[
\vv = \left\{v \in \I_0 : \begin{pmatrix}
0 & v\\
0 & 0
\end{pmatrix} \in \uu\right\} \text{ and } \vv^t = \left\{v \in \I_0 : \begin{pmatrix}
0 & 0\\
v & 0
\end{pmatrix} \in \uu^t\right\}.
\]

\begin{definition}
A \underline{$\Lambda$-lattice} in $Q(\I_0)$ is a finitely generated $\Lambda$-submodule $M$ of $Q(\I_0)$ such that the $Q(\Lambda)$-span of $M$ is equal to $Q(\I_0)$.  If in addition $M$ is a subring of $\I_0$ then we say $M$ is a \underline{$\Lambda$-order}.
\end{definition}

\begin{proof}[Proof of Proposition \ref{open image in product implies main result} when $\alpha = -1$]
We show in Lemmas \ref{u is a lattice} and \ref{ut is a lattice} that $\vv$ and $\vv^t$ are $\Lambda$-lattices in $Q(\I_0)$.  To do this we use the fact that the local Galois representation $\rho_F|_{D_p}$ is indecomposable \cite{Zhao}.  

We then show in Proposition \ref{lattice to ideal} that any $\Lambda$-lattice in $Q(\I_0)$ contains a nonzero $\I_0$-ideal.  Let $\bb$ and $\bb^t$ be nonzero $\I_0$-ideals such that $\bb \subseteq \vv$ and $\bb^t \subseteq \vv^t$.  Let $\Aa_0 = \bb\bb^t$.  Then from the definitions of $\vv, \vv^t$, and $\aL$, we find that 
\[
\aL \supseteq \Aa_0^2\Sl_2(\I_0).
\] 
By Pink's theory it follows that $\G \supseteq \Gamma_{\I_0}(\Aa_0)$.
\end{proof}

Finally, we prove the three key facts used in the proof of Proposition \ref{open image in product implies main result} when $\alpha = -1$.

\begin{lemma}\label{u is a lattice}
With notation as above, $\vv$ is a $\Lambda$-lattice in $Q(\I_0)$.  
\end{lemma}

\begin{proof}
Let $\overline{L} = L(\overline{\G}_{P\I_0})$.  Recall that $L[1]$ surjects onto $\overline{L}[1]$.  Now $\overline{L}[1]$ contains 
\[
\left\{\begin{pmatrix} 
a & 0\\
0 & -a
\end{pmatrix} : a \in \oplus_{\Qq | P} \overline{\Aa}_{\Qq}^2 \right\},
\]
and $\oplus_{\Qq | P} \overline{\Aa}_{\Qq}^2$ is a $\Lambda/P$-lattice in $Q(\I_0/P\I_0)$.  It follows from Nakayama's Lemma that the set of entries in the matrices of $L[1]$ contains a $\Lambda$-lattice $\Aa$ for $Q(\I_0)$.

By a theorem of Zhao \cite{Zhao} we know that $\rho_F|_{D_p}$ is indecomposable.  Hence there is a matrix in the image of $\rho$ whose upper right entry is nonzero.  This produces a nonzero nilpotent matrix in $L_1$.  Taking the Lie bracket of this matrix with a nonzero element of $L[1]$ produces a nonzero nilpotent matrix in $L$ which we will call $\bigl(\begin{smallmatrix} 
0 & v\\
0 & 0
\end{smallmatrix}\bigr)$.  Note that for any $a \in \Aa$ we have
\[
\begin{pmatrix}
0 & 2av\\
0 & 0
\end{pmatrix} = \left[\begin{pmatrix}
a & 0\\
0 & -a 
\end{pmatrix}, \begin{pmatrix} 
0 & v\\
0 & 0
\end{pmatrix} \right] \in L.
\]  
Thus the lattice $\Aa v$ is contained in $\vv$, so $Q(\Lambda)\vv = Q(\I_0)$.  The fact that $\vv$ is finitely generated follows from the fact that $\Lambda$ is noetherian and $\vv$ is contained in the finitely generated $\Lambda$-module $\I_0$.
\end{proof}

\begin{lemma}\label{ut is a lattice}
With notation as above, $\vv^t$ is a $\Lambda$-lattice in $Q(\I_0)$.  
\end{lemma}

\begin{proof}
Let $\overline{c} \in \oplus_{\Qq | P} \overline{\Aa}_{\Qq}^2$.  Since $L[-1]$ surjects to $\overline{L}[-1]$ there is some $\bigl(\begin{smallmatrix} 
0 & b\\
c & 0
\end{smallmatrix}\bigr) \in L$ such that $b \in P\I_0$ and $c \bmod P\I_0 = \overline{c}$.  Since $\vv$ is a $\Lambda$-lattice in $Q(\I_0)$ by Lemma \ref{u is a lattice}, it follows that there is some nonzero $\alpha \in \Lambda$ such that $\alpha b \in \vv$.  

We claim that there is some nonzero $\beta \in \Lambda$ for which $\bigl(\begin{smallmatrix} 
0 & \alpha b\\
\beta c & 0
\end{smallmatrix}\bigr) \in L$.  Assuming the existence of $\beta$, since $\alpha b \in \vv$ it follows that $\beta c \in \vv^t$.  That is, $c \in Q(\Lambda)\vv^t$.  Since $\overline{c}$ runs over $\oplus_{\Qq | P} \overline{\Aa}_{\Qq}^2$, it follows from Nakayama's Lemma that $\vv^t$ is a $\Lambda$-lattice in $Q(\I_0)$.      

To see that $\beta$ exists, recall that $L$ is normalized by the matrix $\textbf{J} = \bigl(\begin{smallmatrix} 
1 + T & 0\\
0 & 1
\end{smallmatrix}\bigr)$.  Thus
\[
\begin{pmatrix}
0 & b\\
c & 0 
\end{pmatrix} + \begin{pmatrix} 
0 & Tb\\
((1 + T)^{-1} - 1)c & 0
\end{pmatrix} = \begin{pmatrix} 
1 + T & 0\\
0 & 1
\end{pmatrix}\begin{pmatrix} 
0 & b\\
c & 0
\end{pmatrix}\begin{pmatrix} 
(1 + T)^{-1} & 0\\
0 & 1
\end{pmatrix} \in L.
\]
Write $\alpha = f(T)$ as a power series in $T$.  Since $(1 + T)^{-1} - 1$ is divisible by $T$, we can evaluate $f$ at $(1 + T)^{-1} - 1$ to get another element of $\Z_p[[T]]$.  Taking $\beta = f((1 + T)^{-1} - 1)$, the above calculation shows that 
\[
\begin{pmatrix}
0 & \alpha b\\
\beta c & 0 
\end{pmatrix} \in L,
\]
as desired.
\end{proof}

\begin{proposition}\label{lattice to ideal}
Every $\Lambda$-lattice in $Q(\I_0)$ contains a nonzero $\I_0$-ideal.
\end{proposition}

\begin{proof}
Let $M$ be a $\Lambda$-lattice in $Q(\I_0)$.  Define 
\[
R = \{x \in \I_0 : xM \subseteq M\}.
\] 
Then $R$ is a subring of $\I_0$ that is also a $\Lambda$-lattice for $Q(\I_0)$.  Thus $R$ is a $\Lambda$-order in $\I_0$, and $M$ is a $R$-module.  Therefore
\[
\cc := \{x \in \I_0 : x\I_0 \subseteq R\}
\]
is a nonzero $\I_0$-ideal.  Note that $Q(R) = Q(\I_0) = Q(\Lambda)M$.  Since $M$ is a finitely generated $\Lambda$-module there is some nonzero $r \in \I_0'$ such that $rM \subseteq R$.  As $rM$ is still a $\Lambda$-lattice for $Q(\I_0)$, by replacing $M$ with $rM$ we may assume that $M$ is a $R$-ideal.

Now consider $\Aa = \cc \cdot (M\I_0)$, where $M\I_0$ is the ideal generated by $M$ in $\I_0$.  Note that $\Aa$ is a nonzero $\I_0$-ideal since both $\cc$ and $M\I_0$ are nonzero $\I_0$-ideals.  To see that $\Aa \subseteq M$, let $x \in \I_0$ and $c \in \cc$.  Then $xc \in R$ by definition of $\cc$.  If $a \in M$ then $xca \in M$ since $M$ is a $R$-ideal.  Thus $xca \in M$, so $\Aa \subseteq M$.
\end{proof}

\begin{remark}
Note that the only property of $\I_0$ that is used in the proof of Proposition \ref{lattice to ideal} is that $\I_0$ is a $\Lambda$-order in $Q(\I_0)$.  Thus, once we have shown that $\rho$ (or $\rho_F$) is $\I_0$-full, it follows that the representation is $R$-full for \textit{any} $\Lambda$-order in $Q(\I_0)$.  In particular, if $\tilde{\I}_0$ is the maximal $\Lambda$-order in $Q(\I_0)$ then $\rho_F$ is $\tilde{\I}_0$-full.
\end{remark}

\section{Open image in product}\label{product image}
The purpose of this section is to prove the following reduction step in the proof of Theorem \ref{main result}.  

\begin{proposition}\label{open image in product}
Assume that $|\F| \neq 3$.  Fix an arithmetic prime $P$ of $\Lambda$.  Assume that for every prime $\Qq$ of $\I_0$ lying over $P$, the image of $\im \rho$ in $\SL_2(\I_0/\Qq)$ is open.  Then the image of $\im \rho$ in $\prod_{\Qq | P} \SL_2(\I_0/\Qq)$ is open in the product topology.
\end{proposition}

Thus if we can show that there is some arithmetic prime $P$ of $\Lambda$ satisfying the hypothesis of Proposition \ref{open image in product}, then combining the above result with Proposition \ref{open image in product implies main result} yields Theorem \ref{main result}.

Fix an arithmetic prime $P$ of $\Lambda$ satisfying the hypothesis of Proposition \ref{open image in product}.  Note that $\Z_p$ does not contain any $p$-power roots of unity since $p > 2$.  Therefore $P = P_{k, 1}$ for some $k \geq 2$.  Recall that $\G = \im \rho \cap \Gamma_{\I_0}(\m_0)$, and write $\overline{\G}$ for the image of $\G$ in $\prod_{\Qq | P}\SL_2(\I_0/\Qq)$.  We begin our proof of Proposition \ref{open image in product} with the following lemma of Ribet which allows us to reduce to considering products of only two copies of $\SL_2$.

\begin{lemma}[Lemma 3.4, \cite{R75}]\label{reduce to two factors}
Let $S_1, \ldots, S_t (t > 1)$ be profinite groups.  Assume for each $i$ that the following condition is satisfied: for each open subgroup $U$ of $S_i$, the closure of the commutator subgroup of $U$ is open in $S_i$.  Let $\cG$ be a closed subgroup of $S = S_1 \times \cdots \times S_t$ that maps to an open subgroup of each group $S_i \times S_j (i \neq j)$.  Then $\cG$ is open in $S$. 
\end{lemma}

Apply this lemma to our situation with $\{S_1, \ldots, S_t\} = \{\SL_2(\I_0/\Qq) : \Qq | P\}$ and $\cG = \overline{\G}$.  The lemma implies that it is enough to prove that for all primes $\Qq_1  \neq \Qq_2$ of $\I_0$ lying over $P$, the image $G$ of $\overline{\G}$ under the projection to $\SL_2(\I_0/\Qq_1) \times \SL_2(\I_0/\Qq_2)$ is open.  We shall now consider what happens when this is not the case.  Indeed, the reader should be warned that the rest of this section is a proof by contradiction.

\begin{proposition}\label{creating new twist}
Let $P$ be an arithmetic prime of $\Lambda$ satisfying the hypotheses of Proposition \ref{open image in product}, and assume $|\F| \neq 3$.  Let $\Qq_1$ and $\Qq_2$ be distint primes of $\I_0$ lying over $P$.  Let $\Pp_i$ be a prime of $\I$ lying over $\Qq_i$.  If $G$ is not open in $\SL_2(\I_0/\Qq_1) \times \SL_2(\I_0/\Qq_2)$ then there is an isomorphism $\sigma : \I_0/\Qq_1 \cong \I_0/\Qq_2$ and a character $\varphi : H \to Q(\I_0/\Qq_2)^\times$ such that
\[
\sigma(a(\ell, f_{\Pp_1})) = \varphi(\ell)a(\ell, f_{\Pp_2})
\]
for all primes $\ell$ for which $\Frob_\ell \in H$.
\end{proposition}

\begin{proof}
Our strategy is to mimic the proof of Theorem 3.5 in \cite{R75}.  Let $G_i$ be the projection of $G$ to $\SL_2(\I_0/\Qq_i)$, so $G \subseteq G_1 \times G_2$.  By hypothesis $G_i$ is open in $\SL_2(\I_0/\Qq_i)$.  Let $\pi_i : G \to G_i$ be the projection maps and set $N_1 = \ker \pi_2$ and $N_2 = \ker \pi_1$.  Though a slight abuse of notation, we view $N_i$ as a subset of $G_i$.  Goursat's Lemma implies that the image of $G$ in $G_1/N_1 \times G_2/N_2$ is the graph of an isomorphism
\[
\alpha : G_1/N_1 \cong G_2/N_2.
\]

Since $G$ is not open in $G_1 \times G_2$ by hypothesis, either $N_1$ is not open in $G_1$ or $N_2$ is not open in $G_2$.  (Otherwise $N_1 \times N_2$ is open and hence $G$ is open.)  Without loss of generality we may assume that $N_1$ is not open in $G_1$.  From the classification of subnormal subgroups of $\SL_2(\I_0/\Qq_1)$ in \cite{T} it follows that $N_1 \subseteq \{\pm 1\}$ since $N_1$ is not open.  If $N_2$ is open in $\SL_2(\I_0/\Qq_2)$ then $\alpha$ gives an isomorphism from either $G_1$ or $\PSL_2(\I_0/\Qq_1)$ to the finite group $G_2/N_2$.  Clearly this is impossible, so $N_2$ is not open in $\SL_2(\I_0/\Qq_2)$.  Again by \cite{T} we have $N_2 \subseteq \{\pm 1\}$.  Recall that $G_i$ comes from $\G = \im \rho \cap \Gamma_{\I_0}(\m_0)$ by reduction.  In particular, $-1 \not\in G_i$ since all elements of $\G$ reduce to the identity in $\SL_2(\F)$.  Thus we must have $N_i = \{1\}$.  Hence $\alpha$ gives an isomorphism $G_1 \cong G_2$.  We note that the Theorem in \cite{T} requires $|\F| \neq 3$.  Our invocation of \cite{T} - here and in the proof of Proposition \ref{transferring congruence subgroup} - is the only reason we assume $|\F| \neq 3$.   

The isomorphism theory of open subgroups of $\SL_2$ over a local ring was studied by Merzljakov in \cite{Me}.  (There is a unique theorem in his paper, and that is the result to which we refer.  His theorem applies to more general groups and rings, but it is relevant in particular to our situation.)  Although his result is stated only for automorphisms of open subgroups, his proof goes through without change for isomorphisms.  His result implies that $\alpha$ must be of the form
\begin{equation}\label{explicit alpha}
\alpha(\textbf{x}) = \eta(\textbf{x})\textbf{y}^{-1}\sigma(\textbf{x})\textbf{y},
\end{equation}
where $\eta \in \Hom(G_1, Q(\I_0/\Qq_2)^\times)$, $\textbf{y} \in \GL_2(Q(\I_0/\Qq_2))$ and $\sigma : \I_0/\Qq_1 \cong \I_0/\Qq_2$ is a ring isomorphism.  By $\sigma(\textbf{x})$ we mean that we apply $\sigma$ to each entry of the matrix $\textbf{x}$.    

For any $\textbf{g} \in G$ we can write $\textbf{g} = (\textbf{x}, \textbf{y})$ with $\textbf{x} \in G_1, \textbf{y} \in G_2$.  Since $G$ is the graph of $\alpha$ we have $\alpha(\textbf{x}) = \textbf{y}$.  By definition of $G$ there is some $h \in H$ such that $\textbf{x} = \Pp_1(\rho(h))$ and $\textbf{y} = \Pp_2(\rho(h))$.  Recall that for almost all primes $\ell$ for which $\Frob_\ell \in H$ we have $\tr(\rho(\Frob_\ell)) = \sqrt{\det \rho_F(\Frob_\ell)}^{-1}a(\ell, F)$.  Furthermore $\det \rho_F(\Frob_\ell) \bmod P = \chi(\ell)\ell^{k - 1}$ since $P = P_{k, 1}$.  Using these facts together with equation \eqref{explicit alpha} we see that for almost any $\Frob_\ell \in H$ we have
\[
\sigma(a(\ell, f_{\Pp_1})) = \varphi(\ell)a(\ell, f_{\Pp_2}),
\]
where
\[
\varphi(\ell) := \eta^{-1}(\Pp_1(\rho(\Frob_\ell)))\frac{\sigma(\sqrt{\chi(\ell)\ell^{k - 1}})}{\sqrt{\chi(\ell)\ell^{k - 1}}},
\]
as claimed.
\end{proof}

To finish the proof of Proposition \ref{open image in product} we need to remove the condition that $\Frob_\ell \in H$ from the conclusion of Proposition \ref{creating new twist}.  That is, we would like to show that there is an isomorphism $\tilde{\sigma} : \I'/\Pp'_1 \cong \I'/\Pp'_2$ extending $\sigma$ and a character $\tilde{\varphi} : G_\Q \to Q(\I'/\Pp'_2)^\times$ extending $\varphi$ such that
\[
\tilde{\sigma}(a(\ell, f_{\Pp_1})) = \tilde{\varphi}(\ell)a(\ell, f_{\Pp_2})
\]
for almost all primes $\ell$.  If we can do this, then applying Theorem \ref{deformation lifting} allows us to lift $\tilde{\sigma}$ to an element of $\Gamma$ that sends $\Pp'_1$ to $\Pp'_2$.  (We also need to verify that $\tilde{\varphi}$ takes values in $\Z_p[\chi]$ in order to apply Theorem \ref{deformation lifting}.)  But this is a contradiction since $\Pp'_1$ and $\Pp'_2$ lie over different primes of $\I_0$.  Hence it follows from Proposition \ref{creating new twist} that $G$ must be open in $\SL_2(\I_0/\Qq_1) \times \SL_2(\I_0/\Qq_2)$ and Lemma \ref{reduce to two factors} implies Proposition \ref{open image in product}. 

We show the existence of $\tilde{\sigma}$ and $\tilde{\varphi}$ using obstruction theory as developed in section 4.3.5 of \cite{MFG}.  For the sake of notation, we briefly recall the theory here; for the proofs we refer the reader to \cite{MFG}.  Let $K$ be a finite extension of $\Q_p$, $n \in \Z^+$, and $r : H \to \GL_n(K)$ be an absolutely irreducible representation.  For all $g \in G_\Q$ define a twisted representation on $H$ by $r^g(h) := r(ghg^{-1})$.  Assume the following condition:
\begin{equation}\label{hypothesis}
r \cong r^g \text{ over } K \text{ for all } g \in G_\Q.
\end{equation}
Under hypothesis \eqref{hypothesis} it can be shown that there is a function $c : G_\Q \to \GL_n(K)$ with the following properties:
\begin{enumerate}
\item $r = c(g)^{-1}r^gc(g)$ for all $g \in G_\Q$;
\item $c(hg) = r(h)c(g)$ for all $h \in H, g \in G_\Q$;
\item $c(1) = 1$.
\end{enumerate}
As $r$ is absolutely irreducible, it follows that $b(g, g') := c(g)c(g')c(gg')^{-1}$ is a $2$-cocycle with values in $K^\times$.  In fact $b$ factors through $\Delta := G_\Q/H$ and hence represents a class in $H^2(\Delta, K^\times)$.  We call this class $\Ob(r)$.  It is independent of the function $c$ satisfying the above three properties.  The class $\Ob(r)$ measures the obstruction to lifting $r$ to a representation of $G_\Q$.  We say a continuous representation $\tilde{r} : G_\Q \to \GL_n(K)$ is an \textit{extension} of $r$ to $G_\Q$ if $\tilde{r}|_H = r$.

\begin{proposition}\label{obstruction theory}
\begin{enumerate}
\item There is an extension $\tilde{r}$ of $r$ to $G_\Q$ if and only if $\Ob(r) = 0 \in H^2(\Delta, K^\times)$.
\item If $\Ob(r) = 0$ and $\tilde{r}$ is an extension of $r$ to $G_\Q$, then all other extensions of $r$ to $G_\Q$ are of the form $\tilde{r} \otimes \psi$ for some character $\psi : \Delta \to K^\times$.
\end{enumerate}
\end{proposition}      

For ease of notation we shall write $K_i = Q(\I/\Pp_i)$ and $E_i = Q(\I_0/\Qq_i)$.  Write $\rho_i : G_\Q \to \GL_2(K_i)$ for $\rho_{f_{\Pp_i}}$.  By Theorem \ref{making rho} we see that $\rho_i|_H$ takes values in $\GL_2(E_i)$.  Proposition \ref{creating new twist} gives an isomorphism $\sigma : E_1 \cong E_2$ and a character $\varphi : H \to E_2^\times$ such that
\[
\tr (\rho_1|_H^\sigma) = \tr(\rho_2|_H \otimes \varphi).
\]

In order to use obstruction theory to show the existence of $\tilde{\sigma}$ and $\tilde{\varphi}$ we must show that all of the representations in question satisfy hypothesis \eqref{hypothesis}.

\begin{lemma}\label{satisfying obstruction theory hypothesis}
Let $L_i$ be a finite extension of $K_i$.  View $\rho_1$ as a representation over $L_1$ and $\rho_2|_H, \rho_1|_H^\sigma, \rho_2|_H \otimes \varphi$, and $\varphi$ as representations over $L_2$.  Then $\rho_i|_H, \rho_1|_H^\sigma, \rho_2|_H \otimes \varphi$, and $\varphi$ all satisfy hypothesis \eqref{hypothesis}.  Furthermore we have $\Ob(\rho_i|_H) = 0, \Ob(\rho_1|_H^\sigma) = \Ob(\rho_2|_H \otimes \varphi)$, and
\[
\Ob(\rho_2|_H \otimes \varphi) = \Ob(\rho_2|_H)\Ob(\varphi) \in H^2(\Delta, (L_2)^\times).
\]
\end{lemma}

\begin{proof}
Recall that a continuous representation of a compact group over a field of characteristic $0$ is determined up to isomorphism by its trace.  Therefore to verify \eqref{hypothesis} it suffices to show that if $r$ is any of the representations listed in the statement of the lemma, then
\[
\tr r = \tr r^g
\]
for all $g \in G_\Q$.  This is obvious when $r$ is $\rho_1|_H$ or $\rho_2|_H$ since both extend to representations of $G_\Q$ and hence 
\[
\tr \rho_i^g(h) = \tr\rho_i(g)\rho_i(h)\rho_i(g)^{-1} = \tr \rho_i(h).
\]
Since $\rho_i$ is an extension of $\rho_i|_H$ and $L_i \supseteq K_i$ we have $\Ob(\rho_i) = 0$.

When $r = \rho_1|_H^\sigma$, let $\tau : K \hookrightarrow \overline{\Q}_p$ be an extension of $\sigma$.  Then $\rho_1^\tau$ is an extension of $\rho_1|_H^\sigma$ and hence we can use the same argument as above to conclude that $\tr \rho_1|_H^\sigma = \tr (\rho_1|_H^\sigma)^g$.  (Note that for this particular purpose, we do not care about the field in which $\tau$ takes values.)  

When $r = \rho_2|_H \otimes \varphi$, recall that $\tr \rho_1|_H^\sigma = \varphi\tr \rho_2|_H$.  Since both $\rho_1|_H^\sigma$ and $\rho_2|_H$ satisfy hypothesis \eqref{hypothesis} so does $\rho_2|_H \otimes \varphi$.  Furthermore, $\tr \rho_1|_H^\sigma = \tr(\rho_2|_H \otimes \varphi)$ implies that $\rho_1|_H^\sigma \cong \rho_2|_H \otimes \varphi$ and hence $\Ob(\rho_1|_H^\sigma) = \Ob(\rho_2|_H \otimes \varphi)$.  

Since $(\rho_1|_H^\sigma)^g \cong \rho_2|_H^g \otimes \varphi^g$ for any $g \in G_\Q$ and since both $\rho_i|_H$ satisfy \eqref{hypothesis} we see that
\begin{equation}\label{epsilon and rho2 together}
\varphi^g \tr \rho_2|_H = \varphi \tr \rho_2|_H.
\end{equation}
Thus if we know $\tr \rho_2|_H$ is nonzero sufficiently often then we can deduce that $\varphi$ satisfies \eqref{hypothesis} .  More precisely, let $m \in \Z^+$ be the conductor for $\varphi$, so $\varphi : (\Z/m\Z)^\times \to \overline{\Q}^\times$.  Then we have a surjection $H \onto \Gal(\Q(\zeta_m)/\Q) \cong (\Z/m\Z)^\times$ with kernel $\kappa$.  Choose a set $S$ of coset representatives of $\kappa$ in $H$, so $H = \sqcup_{s \in S} s\kappa$.  If we can show that $\tr \rho_2(s\kappa) \neq \{0\}$ for all $s \in S$, then it follows from equation \eqref{epsilon and rho2 together} that $\varphi^g = \varphi$ for all $g \in G_\Q$.  Recall that $\rho_2$ is a Galois representation attached to a classical modular form, and so by Ribet \cite{R80}, \cite{R85} and Momose's \cite{Mo} result we know that its image is open.  (See Theorem \ref{Momose} for a precise statement of their result.)  Then the restriction of $\rho_2$ to any open subset of $G_\Q$ also has open image and hence $\tr \rho_2$ is not identically zero.  Each $s\kappa$ is open in $G_\Q$, so $\varphi^g = \varphi$.

Finally, note that if $c : G_\Q \to \GL_2(L_2)$ is a function satisfying conditions 1-3 above for $r = \rho_2|_H$ and $\eta : G_\Q \to L_2^\times$ is a function satisfying conditions 1-3 above for $\varphi$, then $\eta c$ is a function satisfying conditions 1-3 for $\rho_2|_H \otimes \varphi$.  From this it follows that $\Ob(\rho_2|_H \otimes \varphi) = \Ob(\rho_2|_H)\Ob(\varphi)$. 
\end{proof}

With $L_i$ as in the previous lemma, suppose there is an extension $\tilde{\sigma} : L_1 \cong L_2$ of $\sigma$ and an extension $\tilde{\varphi} : G \to L_2^\times$ of $\varphi$.  We now show that this gives us the desired relation among traces.

\begin{lemma}\label{lifts give trace}
If there exists extensions $\tilde{\sigma}$ of $\sigma$ and $\tilde{\varphi}$ of $\varphi$, then there exists a character $\eta : G_\Q \to L_2^\times$ that is also a lift of $\varphi$ such that $\rho_1^{\tilde{\sigma}} \cong \rho_2 \otimes \eta$.
\end{lemma}

\begin{proof}
Note that since $F$ does not have CM, $\rho_1|_H$ and $\rho_2|_H$ are absolutely irreducible by results of Ribet \cite{R77}.  For any absolutely irreducible representation $\pi : G_\Q \to \GL_2(L_2)$ Frobenius reciprocity gives
\begin{equation}\label{Frob rec}
\langle \pi, \Ind_H^{G_\Q}(\rho_1|_H^\sigma) \rangle_{G_\Q} = \langle \pi|_H, \rho_1|_H^\sigma \rangle_H = \langle \pi|_H, \rho_2|_H \otimes \varphi \rangle_H. 
\end{equation}
Thus if $\pi$ is a $2$-dimensional irreducible constituent of $\Ind(\rho_1|_H^\sigma)$ then $\rho_1|_H^\sigma$ is a constituent of $\pi|_H$.  As both are $2$-dimensional, it follows that $\rho_1|_H^\sigma \cong \pi|_H$ and thus $\pi$ is an extension of $\rho_1|_H^\sigma$.  Since $\tilde{\sigma}$ exists by hypothesis, we know that $\rho_1^{\tilde{\sigma}}$ is also an extension of $\rho_1|_H^\sigma$.  

Since $\tilde{\varphi}$ exists by hypothesis, we can take $\pi = \rho_2 \otimes \tilde{\varphi}$.  Then \eqref{Frob rec} implies that $\pi$ is an irreducible constituent of $\Ind_H^G (\rho_1|_H^\sigma)$.  By Proposition \ref{obstruction theory} there is a character $\psi : \Delta \to L_2^\times$ such that $\rho_2 \otimes \tilde{\varphi} \cong \rho_1^{\tilde{\sigma}} \otimes \psi$.  That is,
\[
\rho_1^{\tilde{\sigma}} \cong \rho_2 \otimes (\tilde{\varphi} \psi^{-1}).
\] 
Setting $\eta = \tilde{\varphi}\psi^{-1}$ gives the desired conclusion.
\end{proof}

Finally, we turn to showing the existence of $\tilde{\sigma}$ and $\tilde{\varphi}$.  With notation as in Lemma \ref{satisfying obstruction theory hypothesis}, suppose there exists $\tilde{\sigma}^{-1} : L_2 \cong L_1$ that lifts $\sigma^{-1}$.  Then $\tilde{\sigma}^{-1}$ induces an isomorphism $H^2(\Delta, L_2^\times) \cong H^2(\Delta, L_1^\times)$ that sends $\Ob(\rho_1|_H^\sigma)$ to $\Ob(\rho_1|_H)$.  It follows from Lemma \ref{satisfying obstruction theory hypothesis} that $\Ob(\rho_1|_H^\sigma) = 1$ and hence $\Ob(\rho_2|_H \otimes \varphi) = 1$.  But $1 = \Ob(\rho_2|_H \otimes \varphi) = \Ob(\rho_2|_H)\Ob(\varphi) = \Ob(\varphi)$, and thus we can extend $\varphi$ to $\tilde{\varphi} : G \to L_2^\times$.  
 
The above argument requires that we find $L_i \supseteq K_i$ such that $L_1$ is isomorphic to $L_2$ via a lift of $\sigma$.  We can achieve this as follows.  Let $\tau : K_1 \hookrightarrow \overline{\Q}_p$ be an extension of $\sigma$.  Let $L_2 = K_2\tau(K_1)$.  Let $\tilde{\sigma}^{-1} : L_2 \hookrightarrow \overline{\Q}_p$ be an extension of $\tau^{-1}$ and set $L_1 = \tilde{\sigma}^{-1}(L_2)$.  This construction satisfies the desired properties.  Applying Lemma \ref{lifts give trace} we see that there is a character $\eta : \Delta \to L_2^\times$ such that 
\begin{equation}\label{almost}
\tr \rho_1^{\tilde{\sigma}} = \tr \rho_2 \otimes \eta.
\end{equation}
This is almost what we want.  Note that by \eqref{almost} it follows that $\tilde{\sigma}$ restricts to an isomorphism from $(\I'/\Pp'_1)[\eta]$ to $(\I'/\Pp'_2)[\eta]$.  The only problem is that $\tilde{\sigma}$ may not send $\I'/\Pp'_1$ to $\I'/\Pp'_2$ and $\eta$ may have values in $L_2$ that are not in $(\I'/\Pp'_2)^\times$.  We shall show that this cannot be the case.

Recall that $\chi$ is the Nebentypus of $F$ and $\Pp_1$ and $\Pp_2$ lie over the arithmetic prime $P_{k, 1}$ of $\Lambda$.  Thus for almost all primes $\ell$ we have $\det \rho_i(\Frob_\ell) = \chi(\ell)\ell^{k - 1}$.  Applying this to equation \eqref{almost} we find that
\[
\chi^{\tilde{\sigma}}(\ell)\ell^{k - 1} = \eta^2(\ell)\chi(\ell)\ell^{k - 1}.
\]
Recall that $\chi(\ell)$ is a root of unity and hence $\chi^{\tilde{\sigma}}(\ell)$ is just a power of $\chi(\ell)$.  Thus $\eta^2(\ell) \in \Z_p[\chi] \subseteq \I'/\Pp'_i$ and hence $[(\I'/\Pp'_i)[\eta] : \I'/\Pp'_i] \leq 2$.  Thus we may assume that $L_2 = K_2[\eta]$, which is at most a quadratic extension of $K_2$.  

Note that since $\eta^2$ takes values in $\I'/\Pp'_i$ we can obtain $(\I'/\Pp'_i)[\eta]$ from $\I'/\Pp'_i$ by adjoining a $2$-power root of unity.  (Write $\eta$ as the product of a $2$-power order character and an odd order character and note that any odd order root of unity is automatically a square in any ring in which it is an element.)

\begin{lemma}
We have $(\I'/\Pp'_i)[\eta] = \I'/\Pp'_i$ for $i = 1, 2$.  Therefore $\tilde{\sigma} : \I'/\Pp'_1 \cong \I'/\Pp'_2$ and $\eta$ takes values in $\Z_p[\chi]$.
\end{lemma}  

\begin{proof}
Suppose first that $\I'/\Pp'_2 = (\I'/\Pp'_2)[\eta]$ but $[(\I'/\Pp'_1)[\eta] : \I'/\Pp'_2] = 2$.  Then we have that $\tilde{\sigma} : (\I'/\Pp'_1)[\eta] \cong \I'/\Pp'_2$.  Note that $(\I'/\Pp'_1)[\eta]$ is unramified over $\I'/\Pp'_1$ since it is obtained by adjoining a prime-to-$p$ root of unity (namely a $2$-power root of unity).  Thus the residue field of $(\I'/\Pp'_1)[\eta]$ must be a quadratic extension of the residue field $\F$ of $\I'/\Pp'_1$.  But $\F$ is also the residue field of $\I'/\Pp'_2$ and since $(\I'/\Pp'_1)[\eta] \cong \I'/\Pp'_2$ they must have the same residue field, a contradiction.  Therefore we must have $(\I'/\Pp'_1)[\eta] = \I'/\Pp'_1$.

It remains to deal with the case when $[(\I'/\Pp'_1)[\eta] : \I'/\Pp'_1] = [(\I'/\Pp'_2)[\eta] : \I'/\Pp'_2] = 2$.  As noted above, these extensions must be unramified and hence the residue field of $(\I'/\Pp'_i)[\eta]$ must be the unique quadratic extension $\E = \F[\bbeta]$ of $\F$.  Note that $\tilde{\sigma}$ induces an automorphism $\hat{\sigma}$ of $\E$ that necessarily restricts to an automorphism of $\F$.  From $\chi^{\tilde{\sigma}} = \eta^2\chi$ we find that
\[
\bchi^{\hat{\sigma}} = \bbeta^2\bchi.
\]  
On the other hand $\hat{\sigma}$ is an automorphism of $\F$ and hence is equal to some power of Frobenius.  So we see that for some $s \in \Z$ we have $\bbeta^2 = \bchi^{p^s - 1}$.  Since $p$ is odd, $p^s - 1$ is even and hence $\bbeta^2$ takes values in $\F_p[\bchi^2]$.  Thus $\bbeta$ takes values in $\F_p[\bchi] \subseteq \F$, a contradiction to the assumption that $[\F[\bbeta] : \F] = 2$.  

Since $\eta^2$ takes values in $\Z_p[\chi]$ and $\F_p[\eta] \subseteq \F_p[\chi]$, it follows that in fact $\eta$ must take values in $\Z_p[\chi]$.  Hence we may take $L_i = K_i$ and $\tilde{\sigma} : \I'/\Pp'_1 \cong \I'/\Pp'_2$.
\end{proof}

Finally, we summarize how the results in this section fit together to prove Proposition \ref{open image in product}.

\begin{proof}[Proof of Proposition \ref{open image in product}]
By Lemma \ref{reduce to two factors} it suffices to show that, for any two primes $\Qq_1 \neq \Qq_2$ of $\I_0$ lying over $P_{k, 1}$, the image of $\im \rho$ in $\SL_2(\I_0/\Qq_1) \times \SL_2(\I_0/\Qq_2)$ is open.  Proposition \ref{creating new twist} says that if that is not the case, then there is an isomorphism $\sigma : \I_0/\Qq_1 \cong \I_0/\Qq_2$ and a character $\varphi : H \to Q(\I_0/\Qq_2)^\times$ such that $\tr \rho_{f_{\Pp_1}}|_H^\sigma = \tr \rho_{f_{\Pp_2}}|_H \otimes \varphi$.  The obstruction theory arguments allow us to lift $\sigma$ and $\varphi$ to $\tilde{\sigma} : \I'/\Pp'_1 \cong \I'/\Pp'_2$ and $\tilde{\varphi} : G \to Q(\I/\Pp_2)^\times$ such that $\tr \rho_{f_{\Pp_1}}^{\tilde{\sigma}} = \tr \rho_{f_{\Pp_2}} \otimes \tilde{\varphi}$.  Theorem \ref{deformation lifting} allows us to lift $\tilde{\sigma}$ to an element of $\Gamma$ that sends $\Pp'_1$ to $\Pp'_2$.  But $\Pp'_1$ and $\Pp'_2$ lie over different primes of $\I_0$ and $\Gamma$ fixes $\I_0$, so we reach a contradiction.  Therefore the image of $\im \rho$ in the product $\SL_2(\I_0/\Qq_1) \times \SL_2(\I_0/\Qq_2)$ is open. 
\end{proof}

\section{Proof of Main Theorem}\label{proof of main thm}
In this section we use the compatibility between the conjugate self-twists of $F$ and those of its classical specializations established in section \ref{deformations} to relate $\I_0/\Qq$ to the ring appearing in the work of Ribet \cite{R80}, \cite{R85} and Momose \cite{Mo}.  This allows us to use their results to finish the proof of Theorem \ref{main result}.  

We begin by recalling the work of Ribet and Momose.  We follow Ribet's exposition in \cite{R85} closely.  Let $f = \sum_{n = 1}^\infty a(n, f)q^n$ be a classical eigenform of weight $k$.  Let $K = \Q(\{a(n, f) : n \in \Z^+\})$ with ring of integers $\OK$.  Denote by $\Gamma_f$ the group of conjugate self-twists of $f$.  Let $E = K^{\Gamma_f}$ and $H_f = \cap_{\sigma \in \Gamma_f} \ker \eta_\sigma$.  For any character $\psi$, let $G(\psi)$ denote the Gauss sum of the primitive character of $\psi$.  For $\sigma, \tau \in \Gamma_f$ Ribet defined
\[
c(\sigma, \tau) := \frac{G(\eta_\sigma^{-1})G(\eta_\tau^{-1})}{G(\eta_{\sigma\tau}^{-1})}.
\]
One shows that $c$ is a $2$-cocycle on $\Gamma_f$ with values in $K^\times$.  

Let $\X$ be the central simple $E$-algebra associated to $c$.  Then $K$ is the maximal commutative semisimple subalgebra of  $\X$.  It can be shown that $\X$ has order two in the Brauer group of $E$, and hence there is a $4$-dimensional $E$-algebra $D$ that represents the same element as $\X$ in the Brauer group of $E$.  Namely, if $\X$ has order one then $D = M_2(E)$ and otherwise $D$ is a quaternion division algebra over $E$.  

For a prime $p$, recall that we have a Galois representation
\[
\rho_{f, p} : G_\Q \to \GL_2(\OK_K \otimes_\Z \Z_p)
\]
associated to $f$.  The following theorem is due to Ribet in the case when $f$ has weight $2$ \cite{R80}.  

\begin{theorem}[Momose \cite{Mo}]\label{Momose}
We may view $\rho_{f, p}|_{H_f}$ as a representation valued in $(D \otimes_\Q \Q_p)^\times$.  Furthermore, letting $\n$ denote the reduced norm map on $D$, the image of $\rho_{f, p}|_{H_f}$ is open in
\[
\{x \in (D \otimes_\Q \Q_p)^\times : \n x \in \Q_p^\times\}.
\]
\end{theorem}

In particular, when $D \otimes_\Q \Q_p$ is a matrix algebra, the above theorem tell us that $\im \rho_{f, p}|_{H_f}$ is open in
\[
\{\textbf{x} \in \GL_2(\OK_E \otimes_\Z \Z_p) : \det \textbf{x} \in (\Z_p^\times)^{k - 1}\}.
\]
Let $\p$ be a prime of $\OK_E$ lying over $p$, and let $\rho_{f, \p}$ be the representation obtained by projecting $\rho_{f, p}$ to the $\OK_{E_\p}$-component.  Under the assumption that $D \otimes_\Q \Q_p$ is a matrix algebra Theorem \ref{Momose} implies that $\rho_{f, \p}$ is $\OK_{E_\p}$-full.  Finally, Brown and Ghate proved that if $f$ is ordinary at $p$, then $D \otimes_\Q \Q_p$ is a matrix algebra (Theorem 3.3.1 \cite{BG}).  

Thus, the Galois representation associated to each classical specialization of our $\I$-adic form $F$ is $\OK_{E_\p}$-full with respect to the appropriate ring $\OK_{E_\p}$.  We must show that $E_\p$ is equal to $Q(\I_0/\Qq)$, where $\Qq$ corresponds to $\p$ in a way we will make precise below.   

Recall that we have a fixed embedding $\iota_p : \overline{\Q} \hookrightarrow \overline{\Q}_p$.  Let $\Pp \in \Spec(\I)(\overline{\Q}_p)$ be an arithmetic prime of $\I$, and let $\Qq$ be the prime of $\I_0$ lying under $\Pp$.  As usual, let $\Pp' = \Pp \cap \I'$.  Let $D(\Pp' | \Qq) \subseteq \Gamma$ be the decomposition group of $\Pp'$ over $\Qq$.  Let
\[
K_\Pp = \Q(\{\iota_p^{-1}(a(n, f_\Pp)) : n \in \Z^+\}) \subset \overline{\Q},
\]
and let $\Gamma_\Pp$ be the group of all conjugate self-twists of the classical modular form $f_\Pp$.  
Set $E_\Pp = K_\Pp^{\Gamma_\Pp}$.  Let $\q_\Pp$ be the prime of $K_\Pp$ corresponding to the embedding $\iota_p|_{K_\Pp}$, and set $\p_\Pp = \q_\Pp \cap E_\Pp$.  Let $D(\q_\Pp | \p_\Pp) \subseteq \Gamma_\Pp$ be the decomposition group of $\q_\Pp$ over $\p_\Pp$.  Thus we have that the completion $K_{\Pp, \q_\Pp}$ of $K_\Pp$ at $\q_\Pp$ is equal to $Q(\I/\Pp)$ and $\Gal(K_{\Pp, \q_\Pp}/E_{\Pp, \p_\Pp}) = D(\q_\Pp | \p_\Pp)$.  Thus we may view $D(\q_\Pp | \p_\Pp)$ as the set of all automorphisms of $K_{\Pp, \q_\Pp}$ that are conjugate self-twists of $f_\Pp$.

With this in mind, we see that there is a natural group homomorphism 
\[
\Phi : D(\Pp'|\Qq) \to D(\q_\Pp|\p_\Pp)
\]
since any element of $D(\Pp'|\Qq)$ stabilizes $\Pp'$ and hence induces an automorphism of $Q(\I'/\Pp') = Q(\I/\Pp) = K_{\Pp, \q_\Pp}$.  The induced automorphism will necessarily be a conjugate self-twist of $f_\Pp$ since we started with a conjugate self-twist of $F$.  Thus we get an element of $D(\q_\Pp|\p_\Pp)$.  The main compatibility result is that $\Phi$ is an isomorphism.

\begin{proposition}\label{compatibility}
The natural group homomorphism $\Phi$ is an isomorphism.  Hence $Q(\I_0/\Qq) = E_{\Pp, \p_\Pp}$.
\end{proposition}

\begin{proof}
The fact that $\Phi$ is injective is easy.  Namely, if $\sigma \in D(\Pp'|\Qq)$ acts trivially on $K_{\Pp, \q_\Pp}$ then for almost all $\ell$ we have
\[
a(\ell, f_\Pp) = a(\ell, f_\Pp)^\sigma = \eta_\sigma(\ell)a(\ell, f_\Pp).
\]
Since $F$ (and hence its arithmetic specialization $f_\Pp$) does not have CM it follows that $\eta_\sigma = 1$.  Hence $\sigma = 1$ and $\Phi$ is injective.

To see that $\Phi$ is surjective, let $\sigma \in D(\q_\Pp|\p_\Pp)$.  By Theorem \ref{deformation lifting} we see that there is $\tilde{\sigma} \in \Aut \I'$ that is a conjugate self-twist of $F$ and $\sigma \circ \Pp = \Pp \circ \tilde{\sigma}$.  That is, $\tilde{\sigma} \in D(\Pp'|\Qq)$ and $\Phi(\tilde{\sigma}) = \sigma$.  We have
\[
E_{\Pp, \p_\Pp} = K_{\Pp, \q_\Pp}^{D(\q_\Pp | \p_\Pp)} = Q(\I'/\Pp')^{D(\Pp | \Qq)}.
\]
A general fact from commutative algebra (Theorem V.2.2.2 \cite{B}) tells us that $Q(\I'/\Pp')^{D(\Pp | \Qq)} = Q(\I_0/\Qq)$, as desired.
\end{proof}

\begin{corollary}\label{Ribet/Momose input}
Let $\Qq$ be a prime of $\I_0$ lying over an arithmetic prime of $\Lambda$.  There is a nonzero $\I_0/\Qq$-ideal $\overline{\Aa}_{\Qq}$ such that 
\[
\Gamma_{\I_0/\Qq}(\overline{\Aa}_{\Qq}) \subseteq \im (\rho_F \bmod \Qq\I') \subseteq \prod_{\Pp' | \Qq} \GL_2(\I'/\Pp'),
\]
where the inclusion of $\Gamma_{\I_0/\Qq}(\overline{\Aa}_{\Qq})$ in the product is via the diagonal embedding $\GL_2(\I_0/\Qq) \hookrightarrow \prod_{\Pp' | \Qq}\GL_2(\I'/\Pp')$.  Hence the image of $\im \rho$ in $\SL_2(\I_0/\Qq)$ is open.
\end{corollary}

\begin{proof}
For a prime $\Pp$ of $\I$, write $\OK_\Pp$ for the ring of integers of $E_{\Pp, \p_\Pp}$.  By Theorem \ref{Momose} and the remarks following it, for each prime $\Pp$ of $\I$ lying over $\Qq$ we have $\im \rho_{f_\Pp}$ contains $\Gamma_{\OK_\Pp}(\overline{\Aa}_\Pp)$ for some nonzero $ \OK_\Pp$-ideal $\overline{\Aa}_\Pp$.  While $\I_0/\Qq$ need not be integrally closed, by Proposition \ref{compatibility} we see that $\overline{\Aa}_\Pp \cap (\I_0/\Qq)$ is a nonzero $\I_0/\Qq$-ideal.   

Thus we have
\[
\Gamma_{\I_0/\Qq}(\overline{\Aa}_\Pp \cap \I_0/\Qq) \subseteq \Gamma_{\OK_\Pp}(\overline{\Aa}_\Pp) \subseteq \im \rho_{f_\Pp} = \im \rho_F \bmod \Pp \subseteq \GL_2(\I'/\Pp').
\]
Let $\overline{\Aa}_{\Qq} = \cap_{\Pp | \Qq} \overline{\Aa}_\Pp \cap \I_0/\Qq$.  This is a finite intersection of nonzero $\I_0/\Qq$-ideals and hence is nonzero.  The first statement follows from the above inclusions.  

For the statement about $\rho$, recall that $\rho_F|_{H_0}$ is valued in $\GL_2(\I_0)$ and hence $\im \rho_F|_{H_0} \bmod \Qq$ lies in the diagonally embedded copy of $\GL_2(\I_0/\Qq)$ in $\prod_{\Pp' | \Qq} \GL_2(\I'/\Pp')$.  Since $H$ is open in $G_\Q$ by replacing $\overline{\Aa}_{\Qq}$ with a smaller $\I_0/\Qq$-ideal if necessary, we may assume that $\Gamma_{\I_0/\Qq}(\overline{\Aa}_{\Qq})$ is contained in the image of $\rho_F|_H$ in $\GL_2(\I_0/\Qq)$.  Since $\rho$ and $\rho_F$ are equal on elements of determinant $1$ and $\Gamma_{\I_0/\Qq}(\overline{\Aa}_{\Qq}) \subseteq \SL_2(\I_0/\Qq)$, it follows that $\Gamma_{\I_0/\Qq}(\overline{\Aa}_{\Qq})$ is contained in the image of $\im \rho$ in $\SL_2(\I_0/\Qq)$.  That is, the image of $\im \rho$ in $\SL_2(\I_0/\Qq)$ is open. 
\end{proof}

\begin{proof}[Summary of Proof of Theorem \ref{main result}]
Theorem \ref{making rho}, which will be proved in the next section, allows us to create a representation $\rho : H \to \SL_2(\I_0)$ with the property that if $\rho$ is $\I_0$-full then so is $\rho_F$.  This is important for the use of Pink's theory in section \ref{sufficiency open product image} as well as for the techniques of section \ref{product image}.  Proposition \ref{open image in product implies main result} shows that it is sufficient to prove that the image of $\im \rho$ in $\prod_{\Qq | P} \SL_2(\I_0/\Qq)$ is open for some arithmetic prime $P$ of $\Lambda$.  Proposition \ref{open image in product} further reduces the problem to showing that the image of $\rho$ modulo $\Qq$ is open in $\SL_2(\I_0/\Qq)$ for all primes $\Qq$ of $\I_0$ lying over a fixed arithmetic prime $P$ of $\Lambda$.  

This reduces the problem to studying the image of a Galois representation attached to one of the classical specializations of $F$ (twisted by the inverse square root of the determinant).  Hence we can apply the work of Ribet and Momose, but only after we show that $Q(\I_0/\Qq)$ is the same field that occurs in their work.  This is done in Proposition \ref{compatibility}, though the main input is Theorem \ref{deformation lifting}.   
\end{proof}

\section{Obtaining an $\SL_2(\I_0)$-valued representation}\label{same basis}
In this section we prove Theorem \ref{making rho}.  

\makingrho*
It is well know that so long as $\brho_F$ is absolutely irreducible we may assume that $\rho_F$ has values in $\GL_2(\I')$ and the local representation $\rho_F|_{D_p}$ is upper triangular (Theorem 4.3.2 \cite{GME}).  To show that $\rho_F|_{H_0}$ has values in $\GL_2(\I_0)$ we begin by investigating the structure of $\Gamma$.

\begin{proposition}\label{Gamma is a 2 group}
The group $\Gamma$ is a finite abelian $2$-group.
\end{proposition}

\begin{proof}
Let $S$ be the set of primes $\ell$ for which $a(\ell, F)^\sigma = \eta_\sigma(\ell)a(\ell, F)$ for all $\sigma \in \Gamma$, so $S$ excludes only finitely many primes.  For $\ell \in S$, let
\[
b_\ell := \frac{a(\ell, F)^2}{\det \rho_F(\Frob_\ell)}.
\]
It turns out that $b_\ell \in \I_0$.  To see this, note that since $\brho_F$ is absolutely irreducible, for any $\sigma \in \Gamma$ we have $\rho_F^\sigma \cong \eta_\sigma \otimes \rho_F$ over $\I'$.  Taking determinants we find that $\det \rho_F^{\sigma - 1} = \eta_\sigma^2$.  Thus we have
\[
(a(\ell, F)^\sigma)^2 = \eta_\sigma(\ell)^2a(\ell, F)^2 = \det \rho_F(\Frob_\ell)^{\sigma - 1}a(\ell, F)^2,
\]
from which it follows that $b_\ell^\sigma = b_\ell$.  Solving for $a(\ell, F)$ in the definition of $b_\ell$ we find that 
\[
Q(\I') = Q(\I_0)[\sqrt{b_\ell\det \rho_F(\Frob_\ell)} : \ell \in S].
\]

Recall that for $\ell \in S$ we have $\det \rho_F(\Frob_\ell) = \chi(\ell)\kappa(\langle \ell \rangle)\ell^{-1}$, where $\kappa(\langle \ell \rangle) \in 1 + \m_\Lambda$.  (Currently all that matters is that $\kappa$ is valued in $1 + \m_\Lambda$.  For a precise definition of $\kappa$, see the proof of Theorem \ref{deformation lifting}.)  In particular, $\sqrt{\kappa(\langle \ell \rangle)} \in \Lambda$.  Similarly, we can write $\ell = \langle \ell \rangle \omega(\ell)$ with $\langle \ell \rangle \in 1 + p\Z_p$ and $\omega(\ell) \in \mu_{p - 1}$.  So $\sqrt{\langle \ell \rangle} \in \Lambda$ as well.

Let 
\[
\cK = Q(\I_0)[\sqrt{b_\ell}, \sqrt{\det \rho_F(\Frob_\ell)} : \ell \in S],
\]
which is an abelian extension of $Q(\I_0)$ since it is obtained by adjoining square roots.  The above argument shows that in fact $\cK$ is obtained from $Q(\I_0)[\sqrt{b_\ell} : \ell \in S]$ by adjoining finitely many roots of unity, namely the square roots of the values of $\chi$ and the square roots of $\mu_{p - 1}$.  As odd order roots of unity are automatically squares, we can write $\cK = \Q(\I_0)[\sqrt{b_\ell} : \ell \in S][\mu_{2^s}]$ for some $s \in \Z^+$.  Thus we have
\[
\Gal(\cK/Q(\I_0)) \cong \Gal(Q(\I_0)[\sqrt{b_\ell} : \ell \in S]/Q(\I_0)) \times \Gal(Q(\I_0)[\mu_{2^s}]/Q(\I_0)).
\]
By Kummer theory the first group is an elementary abelian $2$-group.  The second group is isomorphic to $(\Z/2^s\Z)^\times$ and hence is a $2$-group.  As $\Gamma$ is a quotient of $\Gal(\cK/Q(\I_0))$ it follows that $\Gamma$ is a finite abelian $2$-group, as claimed.
\end{proof}

Since $\brho_F$ is absolutely irreducible, it follows from Clifford's Theorem that $\brho_F|_{H_0}$ is semisimple (Theorem 6.5 and Corollary 6.6 \cite{I}).  Furthermore, $\tr \rho_F|_{H_0}$ is $\Gamma$-invariant and the Schur index of $\brho_F$ is one (as is always the case for representations over finite fields).  It follows that there is a representation $\pi : H_0 \to \GL_2(\F^\Gamma)$ that is isomorphic to $\brho_F|_{H_0}$ over $\F$.  This structure gives us three cases, detailed in Proposition \ref{residual cases} below.  We often have to treat the cases separately in what follows.  

Let $D$ be a non-square in $\F^\Gamma$, and let $\E = \F^\Gamma[\sqrt{D}]$ be the unique quadratic extension of $\F^\Gamma$.  Note that since $\Gamma$ is a $2$-group, either $\F = \F^\Gamma$ or $\F^\Gamma \subsetneq \E \subseteq \F$.

\begin{lemma}\label{splitting field for irred pi}
Let $K$ be a field and $\mathcal{S} \subset \GL_n(K)$ a set of nonconstant semisimple operators that can be simultaneously diagonalized over $\overline{K}$.  If $\textbf{y} \in \GL_n(\overline{K})$ such that $\textbf{y}\mathcal{S}\textbf{y}^{-1} \subset \GL_n(K)$, then there is a matrix $\textbf{z} \in \GL_n(K)$ such that $\textbf{z}\mathcal{S}\textbf{z}^{-1} = \textbf{y}\mathcal{S}\textbf{y}^{-1}$.  In particular, if $\pi$ is irreducible over $\F^\Gamma$ but not absolutely irreducible, then $\E$ is the splitting field for $\pi$.  
\end{lemma}

\begin{proof}
Let $\sigma \in G_K := \Gal(\overline{K}/K)$.  Then for any $\textbf{x} \in \mathcal{S}$ we have $\textbf{y}^\sigma \textbf{x} \textbf{y}^{-\sigma} = (\textbf{yxy}^{-1})^\sigma = \textbf{yxy}^{-1}$, so $\textbf{y}^{-1}\textbf{y}^\sigma$ centralizes $\textbf{x}$.  As elements in $\mathcal{S}$ are simultaneously diagonalizable, they have the same centralizer in $\GL_n(\overline{K})$.  Since elements of $\mathcal{S}$ are semisimple, their centralizer is a torus and hence isomorphic to $(\overline{K})^{\oplus n}$.  It's not hard to show that $a : G_K \to (\overline{K}^\times)^{\oplus n}$ given by $\sigma \mapsto \textbf{y}^{-1}\textbf{y}^\sigma$ is a $1$-cocycle.  (Here we view $(\overline{K}^\times)^{\oplus n}$ as a $G_K$-module by letting elements of $G_K$ act component-wise.)  By Hilbert's Theorem 90 we have $H^1(G_K, (\overline{K}^\times)^{\oplus n}) = H^1(G_K, \overline{K}^\times)^{\oplus n} = 0$.  Hence $a$ is a coboundary.  That is, there is some $\alpha \in (\overline{K}^\times)^{\oplus n}$ such that
\[
a_\sigma = \textbf{y}^{-1}\textbf{y}^\sigma = \alpha^{-1}\alpha^\sigma
\]
for all $\sigma \in G_K$.  Thus $(\textbf{y}\alpha^{-1})^\sigma = \textbf{y}\alpha^{-1}$ for all $\sigma \in G_K$, so $\textbf{z} := \textbf{y}\alpha^{-1} \in \GL_n(K)$.  But $\alpha$ commutes with $\mathcal{S}$ and so $\textbf{z}\mathcal{S}\textbf{z}^{-1} = \textbf{y}\mathcal{S}\textbf{y}^{-1}$, as claimed.

To deduce the claim about $\pi$, let $\mathcal{S} = \im \pi$.  If $\pi$ is not absolutely irreducible then there is a matrix $\textbf{y} \in \GL_2(\overline{\F})$ that simultaneously diagonalizes $\mathcal{S}$.  Note that every matrix in $\im \pi$ has eigenvalues in $\E$.  Indeed every matrix has a quadratic characteristic polynomial and $\E$ is the unique quadratic extension of $\F^\Gamma$.  Thus, taking $K = \E$ we see that $\textbf{y}\mathcal{S}\textbf{y}^{-1} \subset \GL_2(K)$.  The first statement of the lemma tells us that $\im \pi$ is diagonalizable over $\E$.  Since $\pi$ is irreducible over $\F^\Gamma$ and $[\E : \F^\Gamma] = 2$, it follows that $\E$ is the smallest extension of $\F^\Gamma$ over which $\im \pi$ is diagonalizable.  
\end{proof}

Let $Z$ be the centralizer of $\im \pi$ in $M_2(\F)$.

\begin{proposition}\label{residual cases}
Assume $\brho_F$ is $H_0$-regular.  Exactly one of the following must happen:
\begin{enumerate}
\item Both $\brho_F|_{H_0}$ and $\pi$ are absolutely irreducible.  In this case $Z$ consists of scalar matrices over $\F$.
\item Neither $\brho_F|_{H_0}$ nor $\pi$ are absolutely irreducible, but $\pi$ is irreducible over $\F^\Gamma$ and $\brho_F|_{H_0}$ is irreducible over $\F$.  In this case $\F = \F^\Gamma$ and we may assume
\[
Z = \left\{\begin{pmatrix} 
\alpha & \beta D\\
\beta & \alpha
\end{pmatrix} : \alpha, \beta \in \F \right\} \cong \E.
\]
\item The representation $\brho_F$ is reducible over $\F$.  In this case we may assume that $Z$ consists of diagonal matrices over $\F$.
\end{enumerate}
\end{proposition}

\begin{proof}
It is clear that exactly one of the three cases must happen.  

To see that $\F = \F^\Gamma$ in case 2, recall that $\E$ is the unique quadratic extension of $\F^\Gamma$.  Since $\Gamma$ is a $2$-group by Lemma \ref{Gamma is a 2 group}, if $\F \neq \F^\Gamma$ then we must have $\F^\Gamma \subsetneq \E \subseteq \F$.  In case 2 this is impossible since $\E$ is the splitting field of $\pi$ so $\E \subseteq \F$ implies that $\pi$ is reducible over $\F$.  But $\brho_F|_{H_0}$ is isomorphic to $\pi$ over $\F$ and in case 2 we have assumed that $\brho_F|_{H_0}$ is irreducible over $\F$.  Hence we must have $\F = \F^\Gamma$.

Note that when $\pi$ is irreducible over $\F$, it follows from Schur's Lemma that $Z$ is a division algebra over $\F$.  As $Z \subseteq M_2(\F)$ we have $1 \leq \dim_\F Z \leq \dim_\F M_2(\F) = 4$.   Furthermore, 
\[
\dim_Z \pi \cdot \dim_\F Z = \dim_\F \pi = 2,
\]
so $\dim_\F Z \leq 2$.  Having $\dim_\F Z = 1$ is equivalent to $\pi$ being absolutely irreducible, which gives case 1.  

In case 2 we have that $Z$ is a quadratic division algebra over $\F$, which is necessarily isomorphic to $\E$.  Hence there is some matrix $\textbf{x} \in Z$ such that $\textbf{x}^2$ is the scalar matrix with $D$ on the diagonal.  Note that both $\textbf{x}$ and $\bigl(\begin{smallmatrix} 
0 & D\\
1 & 0
\end{smallmatrix}\bigr)$ are conjugate to the scalar matrix $\sqrt{D}$ over $\E$.  Hence $\textbf{x}$ is conjugate to $\bigl(\begin{smallmatrix} 
0 & D\\
1 & 0
\end{smallmatrix}\bigr)$ over $\overline{\F}$.  Applying Lemma \ref{splitting field for irred pi} with $S = \{\textbf{x}\}$ and $K = \F$, we see that $\textbf{x}$ is conjugate to $\bigl(\begin{smallmatrix} 
0 & D\\
1 & 0
\end{smallmatrix}\bigr)$ over $\F$, say by a matrix $\textbf{z}$.  Thus by replacing $\pi$ with $\textbf{z}\pi \textbf{z}^{-1}$ we may assume that 
\[
Z = \F\begin{pmatrix} 
1 & 0\\
0 & 1
\end{pmatrix} \oplus \F \begin{pmatrix} 
0 & D\\
1 & 0
\end{pmatrix} = \left\{\begin{pmatrix} 
\alpha & \beta D\\
\beta & \alpha
\end{pmatrix} : \alpha, \beta \in \F \right\}.
\]

Note that $\textbf{y} \in Z$ if and only if $\textbf{x}^{-1}\textbf{yx}$ is in the centralizer of $\brho_F|_{H_0}$ in $\GL_2(\F)$.  If $\brho_F|_{H_0}$ is reducible then its centralizer is either the diagonal matrices of $M_2(\F)$ (in the case when $\brho_F|_{H_0}$ is the sum of two distinct characters) or all of $M_2(\F)$.  Since $\brho_F$ is $H_0$-regular by hypothesis, if $\brho_F|_{H_0}$ is reducible then it must be the sum of two distinct characters.  By conjugating $\brho_F|_{H_0}$ so that it is diagonal we may assume that $Z$ consists of diagonal matrices, so $Z \cong \F \oplus \F$.
\end{proof}

Recall that since $\brho_F$ is absolutely irreducible, for any $\sigma \in \Gamma$ we have $\rho_F^\sigma \cong \eta_\sigma \otimes \rho_F$.  That is, there is some $\textbf{t}_\sigma \in \GL_2(\I')$ such that
\[
\rho_F(g)^\sigma = \eta_\sigma(g)\textbf{t}_\sigma \rho_F(g) \textbf{t}_\sigma^{-1}
\]
for all $g \in G_\Q$.  Then for all $\sigma, \tau \in \Gamma, g \in G_\Q$ we have
\[
\eta_{\sigma\tau}(g)\textbf{t}_{\sigma\tau}\rho_F(g)\textbf{t}_{\sigma\tau}^{-1} = \rho(g)^{\sigma\tau} = \eta_\sigma^\tau(g)\eta_\tau(g)\textbf{t}_\sigma^\tau \textbf{t}_\tau\rho_F(g)\textbf{t}_\tau^{-1}\textbf{t}_\sigma^{-\tau}.
\]
Using the fact that $\eta_{\sigma\tau} = \eta_\sigma^\tau \eta_\tau$ we see that $c(\sigma, \tau) := \textbf{t}_{\sigma\tau}^{-1}\textbf{t}_\sigma^{\tau}\textbf{t}_\tau$ commutes with the image of $\rho_F$.  As $\rho_F$ is absolutely irreducible, $c(\sigma, \tau)$ must be a scalar.  Hence $c$ represents a $2$-cocycle of $\Gamma$ with values in $\I'^\times$.

We will need to treat case 2 from Proposition \ref{residual cases} a bit differently, so we establish notation that will unify the proofs that follow.  For a finite extension $M$ of $\Q_p$, let $\OK_M$ denote the ring of integers of $M$.  Let $K$ be the largest finite extension of $\Q_p$ for which $\OK_K[[T]]$ is contained in $\I'$.  So $K$ has residue field $\F$.  Let $L$ be the unique unramified quadratic extension of $K$.  Write $\J = \Lambda_{\OK_L}[\{a(\ell, F) : \ell \nmid N\}]$.  Note that the residue field of $\J$ is the unique quadratic extension of $\F$.  Let
\[
A = \begin{cases}
\I' & \text{not in case 2}\\
\J & \text{in case 2}.
\end{cases}
\]
Let $\kappa$ be the residue field of $A$, so $\kappa = \E$ in case 2 and $\kappa = \F$ otherwise.  

Since $L$ is obtained from $K$ by adjoining some prime-to-$p$ root of unity, in case 2 it follows that $Q(A)$ is Galois over $Q(\I_0)$ with Galois group isomorphic to $\Gamma \times \Z/2\Z$.  In particular, we have an action of $\Gamma$ on $A$ in all cases.  Let $B = A^\Gamma$.  In case 2, $B$ is a quadratic extension of $A^\Gamma$ and $B \cap \I' = \I_0$.  Otherwise $B = \I_0$.  We may consider the $2$-cocycle $c$ in $H^2(\Gamma, A^\times)$.  

\begin{lemma}\label{c vanishes in H2}
With notation as above, $[c] = 0 \in H^2(\Gamma, A^\times)$.  Thus there is a function $\zeta : \Gamma \to A^\times$ such that $c(\sigma, \tau) = \zeta(\sigma\tau)^{-1}\zeta(\sigma)\zeta(\tau)$ for all $\sigma, \tau \in \Gamma$.
\end{lemma}

\begin{proof}
Consider the exact sequence $1 \to 1 + \m_A \to A^\times \to \kappa^\times \to 1$.  Note that for $j > 0$ we have $H^j(\Gamma, 1 + \m_A) = 0$ since $1 + \m_A$ is a $p$-profinite group for $p > 2$ and $\Gamma$ is a $2$-group by Lemma \ref{Gamma is a 2 group}.  Thus the long exact sequence in cohomology gives isomorphisms
\[
H^j(\Gamma, A^\times) \cong H^j(\Gamma, \kappa^\times)
\]
for all $j > 0$.  Hence it suffices to prove that $[\overline{c}] = 0 \in H^2(\Gamma, \kappa^\times)$.  

Let $\textbf{x} \in \GL_2(\F)$ such that $\pi = \textbf{x}\brho_F|_{H_0} \textbf{x}^{-1}$.  Let $\sigma \in \Gamma$ and $h \in H_0$.  Since $\rho_F^\sigma(h) = \eta_\sigma(h)\textbf{t}_\sigma\rho_F(h)\textbf{t}_\sigma^{-1}$ and $\eta_\sigma(h) = 1$ it follows that $\textbf{x}^\sigma \overline{\textbf{t}}_\sigma \textbf{x}^{-1} \in Z$.

We now split into the cases outlined in Proposition \ref{residual cases}.  Suppose we are in case 1, so $\pi$ is absolutely irreducible.  Then $\textbf{x}^\sigma\overline{\textbf{t}}_\sigma \textbf{x}^{-1}$ must be a scalar in $\F^\times$.  Call it $\bzeta(\sigma)$.  Thus $\overline{\textbf{t}}_\sigma = \bzeta(\sigma)\textbf{x}^{-\sigma}\textbf{x}$.  From this we compute that $\overline{c}(\sigma, \tau) = \bzeta(\sigma\tau)^{-1}\bzeta(\sigma)^\tau\bzeta(\tau)$.  Thus $[\overline{c}] = 0 \in H^2(\Gamma, \F^\times)$.

In case 2, using the description of $Z$ from Proposition \ref{residual cases} we see that $\textbf{x}^\sigma\overline{\textbf{t}}_\sigma \textbf{x}^{-1} = \bigl(\begin{smallmatrix} 
\alpha_\sigma & \beta_\sigma D\\
\beta_\sigma & \alpha_\sigma
\end{smallmatrix}\bigr)$ for some $\alpha_\sigma, \beta_\sigma \in \F$.
This becomes a scalar, say $\bzeta(\sigma) = \alpha_\sigma + \beta_\sigma\sqrt{D}$, over $\E = \kappa$.  Thus $\overline{\textbf{t}}_\sigma = \bzeta(\sigma)\textbf{x}^{-\sigma}\textbf{x}$.  From this we compute that $\overline{c}(\sigma, \tau) = \bzeta(\sigma\tau)^{-1}\bzeta(\sigma)^\tau\bzeta(\tau)$.  Thus $[\overline{c}] = 0 \in H^2(\Gamma, \kappa^\times)$.    

Finally, in case 3 we have that $\textbf{x}^\sigma\overline{\textbf{t}}_\sigma \textbf{x}^{-1}$ is a diagonal matrix.  The diagonal map $\F \hookrightarrow \F \oplus \F$ induces an injection $H^2(\Gamma, \F^\times) \hookrightarrow H^2(\Gamma, \F^\times \oplus \F^\times)$.  The fact that $\textbf{x}^\sigma\overline{\textbf{t}}_\sigma \textbf{x}^{-1}$ is a diagonal matrix allows us to calculate that the image of $[\overline{c}]$ in $H^2(\Gamma, \F^\times \oplus \F^\times)$ is $0$.  Since the map is an injection, it follows that $[\overline{c}] = 0 \in H^2(\Gamma, \F^\times)$, as desired.
\end{proof}

Replace $\textbf{t}_\sigma \in \GL_2(\I')$ by $\textbf{t}_\sigma\zeta(\sigma)^{-1} \in \GL_2(A)$.  Then we still have $\rho_F^\sigma = \eta_\sigma \textbf{t}_\sigma \rho_F \textbf{t}_\sigma^{-1}$, and now $\textbf{t}_{\sigma\tau} = \textbf{t}_\sigma^\tau \textbf{t}_\tau$.  That is, $\sigma \mapsto \textbf{t}_\sigma$ is a nonabelian $1$-cocycle with values in $\GL_2(A)$.  Since $F$ is primitive we have $Q(\I) = Q(\I')$.  Thus by Theorem 4.3.2 in \cite{GME} we see that $\rho_F|_{D_p}$ is isomorphic to an upper triangular representation over $Q(\I')$.  Under the assumptions that $\brho_F$ is absolutely irreducible and $H_0$-regular, the proof of Theorem 4.3.2 in \cite{GME} goes through with $\I'$ in place of $\I$.  That is, $\rho_F|_{D_p}$ is isomorphic to an upper triangular representation over $\I'$.  Let $V = \I'^2$ be the representation space for $\rho_F$ with basis chosen such that 
\[
\rho_F|_{D_p} = \begin{pmatrix} 
\varepsilon & u\\
0 & \delta
\end{pmatrix},
\]
and assume $\bvarepsilon \neq \bdelta$.  Let $V[\varepsilon] \subset V$ be the free direct summand of $V$ on which $D_p$ acts by $\varepsilon$ and $V[\delta]$ be the quotient of $V$ on which $D_p$ acts by $\delta$.  Let $V_A = V \otimes_{\I'} A$.  Similarly for $\lambda \in \{\varepsilon, \delta\}$ let  $V_A[\lambda] := V[\lambda] \otimes_{\I'} A$.  For $\textbf{v} \in V_A$, define
\begin{equation}\label{define action}
\textbf{v}^{[\sigma]} := \textbf{t}_\sigma^{-1}\textbf{v}^\sigma,
\end{equation}
where $\sigma$ acts on $\textbf{v}$ component-wise.  Note that in case 2 we are using the action of $\Gamma$ on $A$ described prior to Lemma \ref{c vanishes in H2}.  

\begin{lemma}\label{new action of Gamma}
For all $\sigma, \tau \in \Gamma$ we have $(\textbf{v}^{[\sigma]})^{[\tau]} = \textbf{v}^{[\sigma\tau]}$, so this defines an action of $\Gamma$ on $V_A$.  Furthermore, this action stabilizes $V_A[\varepsilon]$ and $V_A[\delta]$.    
\end{lemma}

\begin{proof}
The forumula \eqref{define action} defines an action since $\sigma \mapsto \textbf{t}_\sigma$ is a nonabelian $1$-cocycle.  Let $\lambda$ be either $\delta$ or $\varepsilon$.  Let $\textbf{v} \in V_A[\lambda]$ and $\sigma \in \Gamma$.  We must show that $\textbf{v}^{[\sigma]} \in V_A[\lambda]$.  Let $d \in D_p$.  Using the fact that $\textbf{v} \in V_A[\lambda]$ and $\rho_F^\sigma = \eta_\sigma \textbf{t}_\sigma \rho_F \textbf{t}_\sigma^{-1}$ we find that
\[
\rho_F(d)\textbf{v}^{[\sigma]} = \eta_\sigma^{-1}(d)\lambda^\sigma(d)\textbf{v}^{[\sigma]}.
\]
Note that for all $d \in D_p$
\begin{equation}\label{epsilons and deltas}
\begin{pmatrix} 
\varepsilon^\sigma(d) & u^\sigma(d)\\
0 & \delta^\sigma(d)
\end{pmatrix} = \rho_F^\sigma(d) = \eta_\sigma(d)\textbf{t}_\sigma \rho_F(d) \textbf{t}_\sigma^{-1} = \eta_\sigma(d)\textbf{t}_\sigma \begin{pmatrix} 
\varepsilon(d) & u(d)\\
0 & \delta(d)
\end{pmatrix}\textbf{t}_\sigma^{-1}. 
\end{equation}
Using the fact that $\varepsilon \neq \delta$ and that $\rho_F|_{D_p}$ is indecomposable \cite{Zhao} we see that $u/(\varepsilon - \delta)$ cannot be a constant.  (If $u/(\varepsilon - \delta) = \alpha$ is a constant, then conjugating by $\bigl(\begin{smallmatrix} 
1 & \alpha\\
0 & 1
\end{smallmatrix}\bigr)$ makes $\rho_F|_{D_p}$ diagonal.)  Hence $\textbf{t}_\sigma$ must be upper triangular.  Therefore \eqref{epsilons and deltas} implies that $\lambda^\sigma(d) = \eta_\sigma(d)\lambda(d)$, and thus 
\[
\rho_F(d)\textbf{v}^{[\sigma]} = \eta_\sigma^{-1}(d)\lambda^\sigma(d)\textbf{v}^{[\sigma]} = \lambda(d)\textbf{v}^{[\sigma]}.
\]
\end{proof}

We are now ready to show that $\rho_F|_{H_0}$ takes values in $\GL_2(\I_0)$.

\begin{theorem}\label{restricting to H}
Let $\rho_F : G_\Q \to \GL_2(\I')$ such that $\rho_F|_{D_p}$ is upper triangular.  Assume that $\brho_F$ is absolutely irreducible and $H_0$-regular.  Then $\rho_F|_{H_0}$ takes values in $\GL_2(\I_0)$. 
\end{theorem}

\begin{proof}
We have an exact sequence of $A[D_p]$-modules
\begin{equation}\label{exact seq}
0 \to V_A[\varepsilon] \to V_A \to V_A[\delta] \to 0
\end{equation}
that is stable under the new action of $\Gamma$ defined in Lemma \ref{new action of Gamma}.  Tensoring with $\kappa$ over $A$ we get an exact sequence of $\kappa$-vector spaces
\begin{equation}\label{residual exact sequence}
V_\kappa[\bvarepsilon] \to V_\kappa \to V_\kappa[\bdelta] \to 0. 
\end{equation}
Since $V_A[\varepsilon]$ is a direct summand of $V_A$, the first arrow is injective.  Since $V_A[\varepsilon]$ and $V_A$ are free $A$-modules, it follows that $\dim_\kappa V_\kappa[\bvarepsilon] = 1$ and $\dim_\kappa V_\kappa = 2$.  Counting dimensions in \eqref{residual exact sequence} now tells us that $\dim_\kappa V_\kappa[\bdelta] = 1$.  

Going back to the exact sequence \eqref{exact seq} we can take $\Gamma$-invariants since all of the modules are stable under the new action of $\Gamma$.  This gives an exact sequence of $B[D_p \cap H_0]$-modules
\[
0 \to V_A[\varepsilon]^\Gamma \to V_A^\Gamma \to V_A[\delta]^\Gamma \to H^1(\Gamma, V_A[\varepsilon]).
\]
Since $\Gamma$ is a $2$-group by Lemma \ref{Gamma is a 2 group} and $V_A[\varepsilon] \cong A$ is $p$-profinite, we find that $H^1(\Gamma, V_A[\varepsilon]) = 0$.  Tensoring with $\kappa^\Gamma$ over $B$ we get an exact sequence
\[
V_A[\varepsilon]^\Gamma \otimes_B \kappa^\Gamma \to V_A^\Gamma \otimes_B \kappa^\Gamma \to V_A[\delta]^\Gamma \otimes_B \kappa^\Gamma \to 0.
\]

If $\dim_{\kappa^\Gamma} V_A[\lambda]^\Gamma \otimes_B \kappa^\Gamma = 1$ for $\lambda \in \{\varepsilon, \delta\}$, then it follows from Nakayama's Lemma that $V_A[\lambda]^\Gamma$ is a free $B$-module of rank 1.  Hence $V_A^\Gamma$ is a free $B$-module of rank $2$.  In all the cases except case 2, this completes the proof.  In case 2 the above argument tells us that if we view $\rho_F$ as a $\GL_2(A)$-valued representation, then $\rho_F|_{H_0}$ takes values in $\GL_2(B)$.  We know that $\rho_F$ actually has values in $\GL_2(\I')$ and hence $\rho_F|_{H_0}$ has values in $\GL_2(B \cap \I') = \GL_2(\I_0)$.  

Thus we must show that for $\lambda \in \{\varepsilon, \delta\}$ we have $\dim_{\kappa^\Gamma} V_A[\lambda]^\Gamma \otimes_B \kappa^\Gamma = 1$.  Note that $V_A[\lambda]^\Gamma \otimes_B \kappa^\Gamma = V_\kappa[\blambda]^\Gamma$.  It is worth remarking that if $\Gamma$ acts trivially on $\kappa$, then
\[
\dim_{\kappa^\Gamma} V_\kappa[\blambda]^\Gamma = \dim_{\kappa} V_\kappa[\blambda] = 1.
\]
However, this may not be the case.

Write $\overline{\Gamma}$ for the quotient of $\Gamma$ that acts on $\kappa$.  That is, $\overline{\Gamma} = \Gal(\kappa/\kappa^\Gamma)$.  It is cyclic since it is the Galois group of an extension of finite fields.  Let $n = |\overline{\Gamma}|$ and $\sigma \in \overline{\Gamma}$ be a generator.  Since $\dim_\kappa V_\kappa[\blambda] = 1$ we can choose some nonzero $\textbf{v} \in V_\kappa[\blambda]$.  We would like to show that
\[
\sum_{k = 0}^{n - 1} \textbf{v}^{[\sigma^k]} \neq 0
\]
since the right hand side is $\overline{\Gamma}$-invariant.

Since $V_\kappa[\blambda]$ is $1$-dimensional, there is some $\alpha \in \kappa^\times$ such that $\textbf{v}^{[\sigma]} = \alpha \textbf{v}$.  Then we see that for $k \geq 1$
\[
\textbf{v}^{[\sigma^k]} = \left(\prod_{j = 0}^{k - 1}\alpha^{\sigma^j}\right)\textbf{v}.
\]
Thus 
\[
\sum_{k = 0}^{n - 1} \textbf{v}^{[\sigma^k]} = 1 + \sum_{k = 1}^{n - 1}\left(\prod_{j = 0}^{k - 1} \alpha^{\sigma^j} \right)\textbf{v} = \left(1 + \sum_{k = 1}^{n - 1} \alpha^{1 + \sigma + \cdots + \sigma^{k - 1}}\right)\textbf{v}.
\]
If $1 + \sum_{k = 1}^{n - 1}\left(\prod_{j = 0}^{k - 1} \alpha^{\sigma^j} \right) \neq 0$ then we're done.  Otherwise we can change $\textbf{v}$ to $a\textbf{v}$ for any $a \in \kappa^\times$.  It is easy to see that $(a\textbf{v})^{[\sigma]} = a^\sigma\alpha a^{-1}(a\textbf{v})$ and thus changing $\textbf{v}$ to $a\textbf{v}$ changes $\alpha$ to $a^\sigma a^{-1}\alpha$.  So we need to show that there is some $a \in \kappa^\times$ such that $1 + \sum_{k = 1}^{n - 1}\left(\prod_{j = 0}^{k - 1} (a^\sigma a^{-1}\alpha)^{\sigma^j} \right) \neq 0$.

We can rewrite
\[
1 + \sum_{k = 1}^{n - 1}\left(\prod_{j = 0}^{k - 1} (a^\sigma a^{-1}\alpha)^{\sigma^j} \right) = 1 + \sum_{k = 1}^{n - 1} \alpha^{1 + \sigma + \cdots + \sigma^{k - 1}}a^{-1}a^{\sigma^k}.
\]
Thus we are interested in the zeros of the function
\[
f(x) := x + \sum_{k = 1}^{n - 1} \alpha^{1 + \sigma + \cdots + \sigma^{k - 1}}x^{-1}x^{\sigma^k}
\]
on $\kappa$.  By Artin's Theorem on characters (Theorem VI.4.1 \cite{L}), $f$ is not identically zero on $\kappa$.  This shows that $\dim_{\kappa^\Gamma} V_\kappa[\blambda]^\Gamma \geq 1$.

To get equality, let $0 \neq \textbf{w} \in V_\kappa[\blambda]^\Gamma$.  Since $V_\kappa[\blambda]^\Gamma \subseteq V_\kappa[\blambda]$ and $\dim_\kappa V_\kappa[\blambda] = 1$, any element of $V_\kappa[\blambda]^\Gamma$ is a $\kappa$-multiple of $\textbf{w}$.  If $\beta \in \kappa \setminus \kappa^\Gamma$ then $\sigma$ does not fix $\beta$.  Thus
\[
(\beta \textbf{w})^{[\sigma]} = \beta^\sigma \textbf{w}^{[\sigma]} = \beta^{\sigma} \textbf{w} \neq \beta \textbf{w}.
\]
Hence $V_\kappa[\blambda]^\Gamma = \kappa^\Gamma \textbf{w}$ and $\dim_{\kappa^\Gamma} V_\kappa[\blambda]^\Gamma = 1$, as desired.
\end{proof}

Finally, we modify $\rho_F$ to obtain the normalizing matrix $\textbf{j}$ in the last part of Theorem \ref{making rho}.

\begin{lemma}\label{existence of j}
Suppose $\rho_F : G_\Q \to \GL_2(\I')$ such that $\rho_F|_{D_p}$ is upper triangular and $\rho_F|_{H_0}$ is valued in $\GL_2(\I_0)$.  Assume $\brho_F$ is absolutely irreducible and $H_0$-regular.  Then there is an upper triangular matrix $\textbf{x} \in \GL_2(\I_0)$ and roots of unity $\zeta$ and $\zeta'$ such that $\textbf{j} := \bigl(\begin{smallmatrix}
\zeta & 0\\
0 & \zeta'
\end{smallmatrix} \bigr)$ normalizes the image of $\textbf{x}\rho_F \textbf{x}^{-1}$ and $\zeta \not\equiv \zeta' \bmod p$. 
\end{lemma}

\begin{proof}
This argument is due to Hida (Lemma 4.3.20 \cite{GME}).  As $\brho_F$ is $H_0$-regular there is an $h \in H_0$ such that $\bvarepsilon(h) \neq \bdelta(h)$.  Let $\zeta$ and $\zeta'$ be the roots of unity in $\I_0$ satisfying $\zeta \equiv \varepsilon(h) \bmod \m_0$ and $\zeta \equiv \delta(h) \bmod \m_0$.  By our choice of $h$ we have $\zeta \not\equiv \zeta' \bmod p$.  

Let $q = |\F|$.  Then for some $u \in \I_0$
 \[
\lim_{n \to \infty} \rho_F(h)^{q^n} = \begin{pmatrix} 
\zeta & u\\
0 & \zeta'
\end{pmatrix}.
 \]
Conjugating $\rho_F$ by $\bigl(\begin{smallmatrix}
 1 & u/(\zeta - \zeta')\\
 0 & 1
 \end{smallmatrix} \bigr)$ preserves all three of the desired properties, and the image of the resulting representation is normalized by $\textbf{j} = \bigl(\begin{smallmatrix}
\zeta & 0\\ 
0 & \zeta'
\end{smallmatrix} \bigr)$.
\end{proof}


\section{Appendix: Automorphic lifting}
\subsection{Twists as endomorphisms of a Hecke algebra}
In this section we seek to reformulate the existence of conjugate self-twists in terms of commutative diagrams involving certain Hecke algebras.  We use Wiles's interpretation of Hida families \cite{W}.  Namely for a finite extension $\J$ of $\Lambda_\chi$, a formal power series $G = \sum_{i = 1}^\infty a(n, G)q^n$ is a \textit{$\J$-adic cusp form} of level $\Gamma_0(N)$ and character $\chi$ if for almost all arithmetic primes $\Pp$ of $\J$, the specialization of $G$ at $\Pp$ gives the $q$-expansion of an element $g_\Pp$ of $S_k(\Gamma_0(Np^{r(\varepsilon)}), \varepsilon\chi\omega^{-k})$, where $p^{r(\varepsilon)}$ is the order of $\varepsilon$.  (At various points in what follows we will use $\J = \I$ and $\J = \I'$.)  One defines the Hecke operators by their usual formulae on coefficients.  We say $G$ is \textit{ordinary} if it is an eigenform for the Hecke operators whose eigenvalue under $U(p)$ is in $\J^\times$.  Let $\Ss(N, \chi; \J)$ be the $\J$-submodule of $\J[[q]]$ spanned by all $\J$-adic cusp forms of level $\Gamma_0(N)$ and character $\chi$ that are also Hecke eigenforms.  Let $\Ss^{\ord}(N, \chi; \J)$ denote the $\I$-subspace of $\Ss(N, \chi; \J)$ spanned by all ordinary $\J$-adic cusp forms. 

For each Dirichlet character $\psi$, we shall write $c(\psi) \in \Z^+$ for the conductor of $\psi$.

Let $\psi : (\Z/L\Z)^\times \to \overline{\Q}^\times$ be a Dirichlet character.  Let $\eta$ be a primitive Dirichlet character with values in $\Z[\chi]$.  (Every twist character of $F$ has this property by Lemma \ref{powers of chi}.)  Denote by $M(\psi, \eta)$ the least common multiple of $L, c(\psi)^2$, and $c(\psi)c(\eta)$, and let $M$ be any positive integer multiple of $M(\psi, \eta)$.  By Proposition 3.64 \cite{S}, there is a linear map 
\begin{align*}
R_{\psi, \eta}(M) : S_k(\Gamma_0(M), \psi) &\to S_k(\Gamma_0(M), \eta^2\psi)\\
f = \sum_{n = 1}^\infty a(n, f) q^n &\mapsto \eta f = \sum_{n = 1}^\infty \eta(n)a(n, f)q^n.
\end{align*}

We would like to defined a map analogous to $R_{\psi, \eta}(M)$ in the $\J$-adic setting.  

\begin{lemma}\label{analogue to Shimura map}
Let $M$ be a positive integer multiple of $M(\chi, \eta)$.  There is a well defined $\J$-linear map
\begin{align*}
\R_{\chi, \eta}(M) : \Ss(M, \chi; \J) &\to \Ss(M, \eta^2\chi; \J)\\
G = \sum_{n = 1}^\infty a(n, G)q^n &\mapsto \eta G = \sum_{n = 1}^\infty \eta(n)a(n, G)q^n.
\end{align*}
If $p \nmid c(\eta)$ then $\R_{\chi, \eta}(M)$ sends $\Ss^{\ord}(M, \chi; \J)$ to $\Ss^{\ord}(M, \eta^2\chi; \J)$.
\end{lemma}

\begin{proof}
Let $\Pp$ be an arithmetic prime of $\J$, and let $P_{k, \varepsilon}$ be the arithmetic prime of $\Lambda$ lying under $\Pp$.  If $G \in \Ss^{\ord}(M, \chi; \J)$ then
\[
g_\Pp \in S_k(\Gamma_0(Mp^{r(\varepsilon)}), \varepsilon\chi\omega^{-k}).
\]  
Let $\psi = \varepsilon\chi\omega^{-k}$.  One checks easily from the definitions that $M(\psi, \eta) = M(\chi, \eta)p^{r(\varepsilon)}$.  Let $m \in \Z^+$ such that $M = mM(\chi, \eta)$.  Then 
\[
\eta g_\Pp = R_{\psi, \eta}(mM(\psi, \eta))(g_\Pp) \in S_k(\Gamma_0(mM(\psi, \eta)), \eta^2\psi) = S_k(\Gamma_0(Mp^{r(\varepsilon)}), \eta^2\varepsilon\chi\omega^{-k}),
\]
so $\eta G \in \Ss(M, \eta^2\chi; \J)$.

For the statement about ordinarity, we may assume $G$ is a normalized eigenform, so $a(p, G)$ is the eigenvalue of $G$ under the $U(p)$ operator.  If $G$ is ordinary then $a(p, G) \in \J^\times$.  Hence $\eta(p)a(p, G) = a(p, \eta G) \in \J^\times$ if and only if $\eta(p) \neq 0$.
\end{proof}

For the rest of this section, fix a Dirichlet character $\eta$ with values in $\Z[\chi]$.  Let $M$ be a positive integer multiple of $M(\chi, \eta)$.  We wish to unify the classical and $\J$-adic cases in what follows.  Let $A$ be either a ring of integers $\OK$ in a number field containing $\Z[\chi]$ or an integral domain $\J$ that is finite flat over $\Lambda$ and contains $\Z[\chi]$.  We shall write $S(M, \chi; A)$ for either $S_k(\Gamma_0(M), \chi; \OK)$ when $A = \OK$ or $\Ss(M, \chi; \J)$ when $A = \J$.  Let 
\[
r_{\chi, \eta}(M) = \begin{cases}
R_{\chi, \eta}(M) & \text{ when }A = \OK\\
\R_{\chi, \eta}(M) & \text{ when } A = \J.
\end{cases}
\]
Denote by $\leftexp{M}{T(n)}$ the $n$-th Hecke operator on either $S(M, \chi; A)$ or $S(M, \eta^2\chi; A)$.  Note that we use this notation $\leftexp{M}T(n)$ even when $(n, M) > 1$.  The Hecke operators are compatible with $r_{\chi, \eta}(M)$ in the following sense.

\begin{lemma}\label{comm diagram for reta}
For all $n \in \Z^+$ we have 
\[
\leftexp{M}T(n) \circ r_{\chi, \eta}(M) = \eta(n)r_{\chi, \eta}(M) \circ \leftexp{M}T(n).
\]
In particular, both maps are zero when $(n, M) > 1$.
\end{lemma}

\begin{proof}
The classical case follows from the $\J$-adic case by specialization, so we give the proof in the $\J$-adic case.  (Incidentally, the classical case can be proved by exactly the same argument.)

It suffices to prove the lemma when $n = \ell$ is prime.  Let $G \in \Ss(M, \chi; \J)$ and recall that by definition of $\leftexp{M}{T(\ell)}$ we have
\[
a(m, G|\leftexp{M}{T(\ell)}) = a(m\ell, G) + \kappa(\langle \ell \rangle)\chi(\ell)\ell^{-1}a(m/\ell, G),
\]
where $\kappa : 1 + p\Z_p \to \Lambda^\times$ was defined in the proof of Lemma \ref{powers of chi} and $a(m/\ell, G) = 0$ if $\ell \nmid m$.  Applying this formula to $\R_{\chi, \eta}(M)(G) \in \Ss(M, \eta^2\chi; \J)$ we calculate that for all $m \in \Z^+$
\[
a(m, \R_{\chi, \eta}(M)(G)|\leftexp{M}{T(\ell)}) = \eta(\ell)a(m, \R_{\chi, \eta}(M)(G | \leftexp{M}{T(\ell)})).
\]  
This implies that
\[
\leftexp{M}{T(\ell)} \circ \R_{\chi, \eta}(M)(G) = \eta(\ell)\R_{\chi, \eta}(M) \circ \leftexp{M}{T(\ell)}(G),
\]
as desired.
\end{proof}

For the rest of the appendix assume further that $\eta$ is a quadratic character, so $r_{\chi, \eta}(M)$ is an endomorphism of $S(M, \chi; A)$.  Let $h(M, \chi; A)$ be the Hecke algebra of $S(M, \chi; A)$.  Recall that there is a duality given by
\begin{align*}
h(M, \chi; A) &\to \Hom_A(S(M, \chi; A), A)\\
T &\mapsto \langle T, - \rangle,
\end{align*}
where $\langle \leftexp{M}{T(n)}, f \rangle := a(n, f)$ for any normalized Hecke eigenform $f \in S(M, \chi; A)$.  Let $\theta_{\chi, \eta}(M)$ be the $A$-algebra endomorphism of $h(M, \chi; A)$ induced by $r_{\chi, \eta}(M)$ via dualtiy.  By Lemma \ref{analogue to Shimura map} if $p \nmid c(\eta)$ then $\theta_{\chi, \eta}(M)$ restricts to an endomorphism of $\hh^{\ord}(M, \chi; \J)$.  

\begin{lemma}\label{explicit theta}
For all $n \in \Z^+$ we have
\[
\theta_{\chi, \eta}(M)(\leftexp{M}{T(n)}) = \eta(n)\leftexp{M}{T(n)}.
\]
\end{lemma}

\begin{proof}
By definition $\theta_{\chi, \eta}(M)$ is the map that makes the following diagram commute.
\[\xymatrix@=40pt@R=30pt{
	h(M, \chi; A)\ar@{<->}[r]\ar@{->}[d]_{\theta_{\chi, \eta}(M)} &\Hom_A(S(M, \chi; A), A)\ar@{->}[d]_{r_{\chi, \eta}(M)^*} \\
	h(M, \chi; A)\ar@{<->}[r] & \Hom_A(S(M, \chi; A), A)}
\]
Note that $\leftexp{M}{T(n)}$ corresponds to $\langle \leftexp{M}{T(n)}, - \rangle$ under duality, and $r_{\chi, \eta}(M)^*(\langle \leftexp{M}{T(n)}, - \rangle) = \langle \leftexp{M}{T(n)}, - \rangle \circ r_{\chi, \eta}(M)$.  Using the formula for the action of $\leftexp{M}{T(n)}$ on $q$-expansions as in Lemma \ref{comm diagram for reta} together with the definition of $r_{\chi, \eta}(M)$ yields
\[
\langle \leftexp{M}{T(n)}, - \rangle \circ r_{\chi, \eta}(M)(f) = \eta(n)\langle \leftexp{M}{T(n)}, f \rangle
\]
for all $f \in S(M, \chi)$.  Thus $\langle \leftexp{M}{T(n)}, - \rangle \circ r_{\chi, \eta}(M) = \eta(n)\langle \leftexp{M}{T(n)}, - \rangle$ which corresponds to $\eta(n)\leftexp{M}{T(n)}$ under duality.  Thus $\theta_{\chi, \eta}(M)(\leftexp{M}{T(n)}) = \eta(n)\leftexp{M}{T(n)}$, as claimed.
\end{proof}

\begin{lemma}\label{constructing fM'}
Let $f \in S(N, \chi; A)$ be an eigenform and $M$ a positive integer multiple of $N$.  There is an eigenform $f_M \in S(M, \chi; A)$ such that $f_M|\leftexp{M}{T(n)} = 0$ for all $n$ such that $(n, M/N) > 1$ and $f_M$ has the same eigenvalues as $f$ for all $\leftexp{M}{T(n)}$ with $(n, M/N) = 1$.
\end{lemma}

\begin{proof}
Write $M/N = \ell_1\ldots\ell_t$ for not necessarily distinct primes $\ell_i$.  By induction on $t$ it suffices to show that we can construct an eigenform $f_{N\ell_1} \in S(N\ell_1, \chi)$ with $f_{N\ell_1}|\leftexp{N\ell_1}{T(\ell_1)} = 0$ and $f_{N\ell_1}$ having the same eigenvalues as $f$ for all primes $\ell \neq \ell_1$.

Let $\lambda_1$ be the eigenvalue of $f$ under $\leftexp{N}{T(\ell_1)}$, so 
\[
f | \leftexp{N}{T(\ell_1)} = \lambda_1f.
\]
If $\lambda_1 = 0$ then just viewing $f \in S(N\ell_1, \chi; A)$ has all the desired properties and we may take $f_{N\ell_1} = f$.  Otherwise, define $f_{N\ell_1} = f - \lambda_1f|[\ell_1]$ where $(f|[\ell_1])(z) := f(\ell_1z)$.  It is well known (and can be checked by a calculation with $q$-expansions) that $f|[\ell_1]|\leftexp{N\ell_1}{T(\ell_1)} = f|\leftexp{N}{T(\ell_1)}$.  This implies that $f_{N\ell_1} | \leftexp{N\ell_1}{T(\ell_1)} = 0$.  For $\ell \neq \ell_1$ one can check that 
\[
\leftexp{N\ell_1}{T(\ell)} \circ [\ell_1] = [\ell_1] \circ \leftexp{N}{T(\ell)}.
\]
From this it follows that $f_{N\ell_1}$ and $f$ have the same eigenvalues for $\leftexp{N\ell_1}{T(\ell)}$ for all primes $\ell \neq \ell_1$, as desired.
\end{proof}

We are interested in describing conjugate self-twists of an eigenform $f \in S(N, \chi; A)$.  Let $A'$ be the subalgebra of $A$ generated by $\{a(\ell, f) : \ell \nmid N \}$ over either $\Z[\chi]$ if $A = \OK$, or over $\Lambda_\chi$ when $A = \I$.  Note that if $f$ is a newform then $Q(A) = Q(A')$.  If $N^2 | M$ then the eigenform $f_M$ from Lemma \ref{constructing fM'} is an element of $S(M, \chi; A')$.  Write $\lambda_{f_M} : h(M, \chi; A') \to A'$ for the $A'$-algebra homomorphism corresponding to $f_M$.  That is, $\lambda_{f_M}(\leftexp{M}{T(n)}) = a(n, f_M)$ for all $n \in \Z^+$.  

\begin{proposition}\label{Iadic twisting}
Let $f \in S(N, \chi; A)$ be primitive and let $\eta$ be a primitive quadratic character.  Let $M = c(\eta)N^2$.  Then $f$ has a conjugate self-twist with character $\eta$ if and only if there is an automorphism $\sigma$ of $A'$ making the following diagram commute.
\[\xymatrix@=40pt@R=30pt{
	h(M, \chi; A')\ar@{->}[r]^{\theta_{\chi, \eta}(M)}\ar@{->}[d]_{\lambda_{f_M}} &h(M, \chi; A')\ar@{->}[d]_{\lambda_{f_M}} \\
	A'\ar@{-->}[r]^{\exists \sigma} & A'}
\]
\end{proposition}

\begin{proof}
Let $f \in S(N, \chi; A)$ be an eigenform.

First suppose that we are given the above diagram for some $\sigma \in \Aut A'$.  Let $\ell$ be a prime not dividing $M$.  Then from the diagram and the definition of $f_M$ we have $\sigma(a(\ell, f)) = \lambda_{f_M} \circ \theta_{\chi, \eta}(M)(\leftexp{M}{T(\ell)})$.  From the description of $\theta_{\chi, \eta}(M)$ in Lemma \ref{explicit theta} and the fact that $\eta$ takes values in $A'$ and $\lambda_{f_{M}}$ is an $A'$-algebra homomorphism, we see that 
\[
\sigma(a(\ell, f)) = \eta(\ell)\lambda_{f_M}(\leftexp{M}{T(\ell)}) = \eta(\ell)a(\ell, f).
\]
Thus $\sigma$ is a conjugate self-twist of $f$ with character $\eta$.

Conversely assume that there is a conjugate self-twist $\sigma$ of $f$ with character $\eta$.  Then we have that $\rho_f^\sigma \cong \rho_f \otimes \eta$.  Since $\rho_f$ is unramified away from $N$ it follows that the only primes $\ell$ for which
\[
\sigma(a(\ell, f)) \neq \eta(\ell)a(\ell, f)
\]
are those dividing $N$.  We need only check that the diagram commutes for $\leftexp{M}{T(\ell)}$ for all primes $\ell$.  If $\ell | M/N$ then both compositions are zero.  If $\ell \nmid M/N = c(\eta)N$ then using the definition of $\theta_{\chi, \eta}(M)$ and $\lambda_{f_M}$ we see that
\[
\sigma \circ \lambda_{f_M}\left(\leftexp{M}{T(\ell)}\right) = \lambda_{f_M} \circ \theta_{\chi, \eta}(M)\left(\leftexp{M}{T(\ell)}\right),
\]
as desired.
\end{proof}

As the previous proposition shows, we will want  $\eta$ to be a twist character of $F$ or one of its specializations.  We had to impose the condition that $\eta$ be quadratic.  By Lemma \ref{powers of chi}, this can be achieved by assuming that the Nebentypus $\chi$ is quadratic.  We now show that in fact, for applications to fullness we need only assume that the order of $\chi$ is not divisible by four.

Note that the ring $\I_0$ depends on $F$.  However, let $\psi$ be a character and $M$ a positive integer multiple of $M(\chi, \psi)$.  Then $\R_{\chi, \psi}(M)(F)$ has the same group of conjugate self-twists as that of $F$, and thus also the same fixed ring $\I_0$.  Indeed, if $\sigma$ is a conjugate self-twist of $F$ with character $\eta$, then a straightforward calculation shows that $\psi^\sigma\eta\psi^{-1}$ is the twist character of $\sigma$ on $\R_{\chi, \psi}(M)(F)$.  

\begin{proposition}\label{preserving big image}
There is a Dirichlet character $\psi$ with values in $\Z[\chi]$ such that, if $M$ is any positive integer multiple of $M(\chi, \psi)$, then the Nebentypus of $\R_{\chi, \psi}(M)(F)$ has order a power of $2$.  Furthermore, $\rho_F$ is $\I_0$-full if and only if $\rho_{\psi F}$ is $\I_0$-full.  
\end{proposition}

\begin{proof}
It is well known (see, for example, Proposition 3.64 \cite{S}) that the Nebentypus of $\psi F$ is $\psi^2\chi$.  Write $\chi = \chi_2\xi$, where $\chi_2$ is a character whose order is a power of $2$ and $\xi$ is an odd order character.  Clearly $\xi$ takes values in $\Z[\chi]$. Let $2n - 1$ denote the order of $\xi$.  Then $\xi^{2n} = \xi$, so taking $\psi = \xi^{-n}$ we see that $\psi^2\chi = \chi_2\xi^{-2n}\xi = \chi_2$ is a character whose order is a power of $2$.

Suppose that $\rho_{\psi F}$ is $\I_0$-full.  Since $\psi$ is a finite order character, $\ker \psi$ is an open subgroup of $G_\Q$.  Thus $\rho_{\psi F}|_{\ker \psi}$ is also $\I_0$-full.  Note that $\rho_{\psi F}|_{\ker \psi} = \rho_{F}|_{\ker \psi}$.  Thus $\rho_F$ is $\I_0$-full.
\end{proof}

Using Proposition \ref{preserving big image} we may assume that the Nebentypus $\chi$ of $F$ has order a power of $2$.  If we want to use the automorphic lifting techniques developed in the previous section, we must further assume that 
\begin{equation}\label{quadratic Nebentypus}
\chi \text{ has order two.}
\end{equation}
This assumption will be in place for the rest of the appendix.  

\subsection{Reinterpreting $\I$-adic conjugate self-twists}
Fix an arithmetic prime $\Qq$ of $\I_0$ lying over $P_{k, \varepsilon}$.  The total ring of fractions $Q(\I'/\Qq\I')$ of $\I'/\Qq\I'$ breaks up as a finite product of fields indexed by the primes of $\I'$ lying over $\Qq$.  Namely 
\[
Q(\I'/\Qq\I') \cong \prod_{\Pp' | \Qq} Q(\I'/\Pp').  
\]
(This relies two facts.  First $\Qq$ is an arithmetic prime and hence unramified in $\I'$.  Secondly the cokernel of $\I'/\Qq\I' \hookrightarrow \prod_{\Pp' | \Qq} \I'/\Pp'$ is finite since $\dim \I' = 2$.)  Let $\Gamma_{\Qq}$ be the group of all automorphisms $\sigma$ of $Q(\I'/\Qq\I')$ for which there is a Dirichlet character $\eta_\sigma$ such that 
\[
\sigma(a(\ell, F) + \Qq\I') = \eta_\sigma(\ell)a(\ell, F) + \Qq\I'
\] 
for all but finitely many primes $\ell$.  Since $\Qq \subseteq \I_0$ and $\I_0$ is fixed by $\Gamma$, elements of $\Gamma$ preserve $\Qq\I'$.  Hence there is a natural group homomorphism
\[
\Psi : \Gamma \to \Gamma_{\Qq}
\]
by letting $\sigma \in \Gamma$ act on $Q(\I'/\Qq\I')$ via $\sigma(a(\ell, F) + \Qq\I') := \sigma(a(\ell, F)) + \Qq\I'$.  While we expect that $\Psi$ is an isomorphism in general, our techniques only allow us to lift certain elements of $\Gamma_{\Qq}$ to $\Gamma$.  Let
\begin{align*}
\Gamma^{2, p} &= \{\sigma \in \Gamma : \eta_\sigma^2 = 1 \text{ and } p \nmid c(\eta_\sigma)\}\\
\Gamma_{\Qq}^{2, p} &= \{\sigma \in \Gamma_{\Qq} : \eta_\sigma^2 = 1 \text{ and } p \nmid c(\eta_\sigma)\}.
\end{align*}
It is easy to check that $\Gamma^{2, p}$ is a subgroup of $\Gamma$ and $\Gamma_{\Qq}^{2, p}$ is a subgroup of $\Gamma_{\Qq}$.  Note that under assumption \eqref{quadratic Nebentypus} the condition $\eta_\sigma^2 = 1$ is automatic.  We shall show that $\Psi|_{\Gamma^{2, p}} : \Gamma^{2, p} \to \Gamma_{\Qq}^{2, p}$ is an isomorphism.  

\begin{proposition}\label{twist isomorphism}
The homomorphism $\Psi : \Gamma \to \Gamma_{\Qq}$ is injective.  Furthermore, $\Psi(\Gamma^{2, p}) = \Gamma_{\Qq}^{2, p}$.
\end{proposition}

\begin{proof}
Suppose $\sigma \in \Gamma$ such that $\Psi(\sigma)$ is trivial.  Thus for almost all primes $\ell$ we have
\[
a(\ell, F) + \Qq\I' = \sigma(a(\ell, F)) + \Qq\I' = \eta_\sigma(\ell)a(\ell, F) + \Qq\I'.
\]
Recall that $\Qq\I' = \cap_{\Pp' | \Qq} \Pp'$, so for all primes $\Pp$ of $\I$ lying over $\Qq$ and almost all rational primes $\ell$ we have
\[
a(\ell, f_\Pp) = \eta_\sigma(\ell)a(\ell, f_\Pp).
\]
Since $f_\Pp$ is a non-CM form it follows that $\eta_\sigma$ must be the trivial character.  Therefore $\sigma = 1$ and $\Psi$ is injective.

Now we show that we can lift elements of $\Gamma_{\Qq}^{2, p}$.  Let $\sigma \in \Gamma_{\Qq}^{2, p}$, so for almost all primes $\ell$,
\[
\sigma(a(\ell, F) + \Qq\I') = \eta_\sigma(\ell)a(\ell, F) + \Qq\I'.
\]  
Let $M = c(\eta_\sigma)N^2$ and consider the map $\theta_{\chi, \eta_\sigma}(M)$ defined before Lemma \ref{explicit theta} with $A = \I$.  Since $p \nmid c(\eta_\sigma)$ we know by Lemma \ref{analogue to Shimura map} that $\theta_{\chi, \eta_\sigma}(M)$ restricts to an endomorphism: 
\begin{align*}
\theta_{\chi, \eta_\sigma}(M) : \hh^{\ord}(M, \chi; \I) & \to \hh^{\ord}(M, \chi; \I)\\
\leftexp{M}{T(n)} &\mapsto \eta_\sigma(n)\leftexp{M}{T(n)}.
\end{align*}
Note that $\theta_{\chi, \eta_\sigma}(M)$ restricts to an endomorphism of $\hh^{\ord}(M, \chi; \I')$.  By Proposition \ref{Iadic twisting} it suffices to show that $\theta_{\chi, \eta_\sigma}(M)$ preserves the $\I'$-component of the Hecke algebra $\hh^{\ord}(M, \chi; \I')$.  The induced map $\theta_{\chi, \eta_\sigma}(M)^*$ on spectra must send irreducible components to irreducible components.  Furthermore since $\sigma \in \Gamma_{\Qq}$ we have the following commutative diagram:
\[\xymatrix@=40pt@R=30pt{
	\hh^{\ord}(M, \chi; \I')\ar@{->}[r]^{\theta_{\chi, \eta_\sigma}(M)}\ar@{->}[d]_{\lambda_{F_M} \bmod \Qq\I'} & \hh^{\ord}(M, \chi; \I')\ar@{->}[d]^{\lambda_{F_M} \bmod \Qq\I'} \\
	\I'/\Qq\I'\ar@{->}[r]^\sigma & \I'/\Qq\I'}
\]
That is, $\theta_{\chi, \eta_\sigma}(M)^*$ maps set of points of $\Spec \I'$ lying over $\Qq$ to itself.  Hence the two irreducible components $\Spec \I'$ and $\theta_{\chi, \eta_\sigma}(M)^*(\Spec \I')$ of $\Spec \hh^{\ord}(M, \chi; \I')$ have nonempty intersection. (Namely, they intersect in some points of $\I'$ lying over $\Qq$.)  Since $\Spec \hh^{\ord}(M, \chi; \I')$ is \'etale over $\Spec \Lambda$ at arithmetic points (see Proposition 3.78 \cite{HMI}) and $\Qq$ is arithmetic, we must have $\theta_{\chi, \eta_\sigma}(M)^*(\Spec \I') = \Spec \I'$.  That is, there is an automorphism $\tilde{\sigma} : \I' \to \I'$ such that the following diagram commutes:
\[\xymatrix@=40pt@R=30pt{
	\hh^{\ord}(M, \chi; \I')\ar@{->}[r]^{\theta_{\chi, \eta_\sigma}(M)}\ar@{->}[d]_{\lambda_{F_M}} & \hh^{\ord}(M, \chi; \I')\ar@{->}[d]_{\lambda_{F_M}} \\
	\I'\ar@{->}[r]^{\tilde{\sigma}}\ar@{->}[d] & \I'\ar@{->}[d]\\
	\I'/\Qq\I'\ar@{->}[r]^\sigma & \I'/\Qq\I'}
\]
By Lemma \ref{explicit theta} and the definition of $\lambda_{F_M}$ we see that $\tilde{\sigma} \in \Gamma$.  As the lower square of the above diagram commutes, it follows that $\Psi(\tilde{\sigma}) = \sigma$, as desired.
\end{proof}

\begin{remark}
With notation as in the above proof, if $p | c(\eta_\sigma)$ then the irreducible component $\theta_{\chi, \eta_\sigma}(M)^*(\Spec \I')$ is no longer ordinary.  Indeed, it corresponds to $\R_{\chi, \eta_\sigma}(M)(F)$ which has infinite slope.  Thus we would need an \'etaleness result for an eigencurve that includes the infinite slope modular forms in order for the above proof to allow us to lift $\sigma$ to $\Gamma$.    
\end{remark}

\subsection{Identifying $\I$-adic and classical decomposition groups}
We briefly recall the notation introduced in section \ref{proof of main thm}.  We have a fixed embedding $\iota_p : \overline{\Q} \hookrightarrow \overline{\Q}_p$.  Let $\Pp_0 \in \Spec(\I)(\overline{\Q}_p)$ be an arithmetic prime of $\I$, and let $\Qq$ be the prime of $\I_0$ lying under $\Pp_0$.  Let $D(\Pp_0' | \Qq) \subseteq \Gamma$ be the decomposition group of $\Pp_0'$ over $\Qq$.  Let
\[
K_{\Pp_0} = \Q(\{\iota_p^{-1}(a(n, f_{\Pp_0})) : n \in \Z^+\}) \subset \overline{\Q},
\]
and let $\Gamma_{\Pp_0}$ be the group of all conjugate self-twists of the classical modular form $f_{\Pp_0}$.  As in the previous section, define
\[
\Gamma_{\Pp_0}^{2, p} = \{\sigma \in \Gamma_{\Pp_0} : \eta_\sigma^2 = 1 \text{ and } p \nmid c(\eta_\sigma)\}.
\]  
Set $E_{\Pp_0} = K_{\Pp_0}^{\Gamma_{\Pp_0}}$.  Let $\q_{\Pp_0}$ be the prime of $K_{\Pp_0}$ corresponding to the embedding $\iota_p|_{K_{\Pp_0}}$, and set $\p_{\Pp_0} = \q_{\Pp_0} | \cap E_{\Pp_0}$.  Let $D(\q_{\Pp_0} | \p_{\Pp_0}) \subseteq \Gamma_\Pp$ be the decomposition group of $\q_{\Pp_0}$ over $\p_{\Pp_0}$.  Thus we have that the completion $K_{\Pp_0, \q_{\Pp_0}}$ of $K_{\Pp_0}$ at $\q_{\Pp_0}$ is equal to $Q(\I/\Pp_0)$ and $\Gal(K_{\Pp_0, \q_{\Pp_0}}/E_{\Pp_0, \p_{\Pp_0}}) = D(\q_{\Pp_0} | \p_{\Pp_0})$.  Thus we may view $D(\q_{\Pp_0} | \p_{\Pp_0})$ as the set of all automorphisms of $K_{\Pp_0, \q_{\Pp_0}}$ that are conjugate self-twists of $f_{\Pp_0}$.

Let
\[
\Phi : D(\Pp_0' | \Qq) \to D(\q_{\Pp_0} | \p_{\Pp_0})
\]
be the natural homomorphism defined in section \ref{proof of main thm}.  We saw that $\Phi$ is an isomorphism in section \ref{proof of main thm}.  In this section we give a second proof that 
\[
D(\q_{\Pp_0} | \p_{\Pp_0}) \cap \Gamma_{\Pp_0}^{2, p} \subseteq \im \Phi.
\]

\begin{theorem}\label{identify decomp gps}
We have $D(\q_{\Pp_0} | \p_{\Pp_0}) \cap \Gamma_{\Pp_0}^{2, p} \subseteq \im \Phi$.
\end{theorem}

\begin{proof}
Let $\sigma \in D(\q_{\Pp_0} | \p_{\Pp_0}) \cap \Gamma_{\Pp_0}^{2, p}$.  For any prime$\Pp'$ of $\I'$ lying over $\Qq$, there is some $\gamma_{\Pp'} \in \Gamma$ such that $\gamma_{\Pp'}(\Pp_0') = \Pp'$.  Then $\gamma_{\Pp'}$ induces an automorphism $\bgamma_{\Pp'}$ of $\overline{\Q}_p$ such that $\bgamma_{\Pp'}(a(\ell, f_\Pp)) = \eta_{\gamma_{\Pp'}}(\ell)a(\ell, f_{\Pp_0})$ for almost all primes $\ell$.  Then 
\[
\bgamma_{\Pp'}^{-1} \circ \sigma \circ \bgamma_{\Pp'} \in D(\q_{\Pp'} | \p_{\Pp'}).
\]  
In fact, we can compute the action of this element explicitly.  This computation makes use of the fact that all twist characters are quadratic and hence their values are either $\pm 1$.  In particular, they are fixed by all automorphisms in question.  For almost all primes $\ell$ we have
\begin{equation}\label{computing action explicitly}
\bgamma_{\Pp'}^{-1} \circ \sigma \circ \bgamma_{\Pp'}(a(\ell, f_\Pp)) = \eta_\sigma(\ell)a(\ell, f_\Pp).
\end{equation}
This shows that the automorphism $\bgamma_{\Pp'}^{-1} \circ \sigma \circ \bgamma_{\Pp'}$ is independent of the choice of $\gamma_{\Pp'}$ sending $\Pp_0'$ to $\Pp'$.  

We can now put all of these automorphisms $\bgamma_{\Pp'}^{-1} \circ \sigma \circ \bgamma_{\Pp'}$ together to obtain an automorphism 
\[
\pi := \prod_{\Pp' | \Qq}\bgamma_{\Pp'}^{-1}\circ \sigma \circ \bgamma_{\Pp'}
\] 
of $Q(\I'/\Qq\I') \cong \prod_{\Pp' | \Qq} Q(\I'/\Pp') = \prod_{\Pp' | \Qq} K_{\Pp, \q_\Pp}$ by simply letting each $\bgamma_{\Pp'}^{-1} \circ \sigma \circ \bgamma_{\Pp'}$ act on $K_{Pp, \q_\Pp}$.  By equation \eqref{computing action explicitly} we see that $\pi$ is in fact an element of $\Gamma_{\Qq}^{2, p}$.  Thus by Proposition \ref{twist isomorphism} it follows that $\pi$, and hence $\sigma$, can be lifted to an element $\tilde{\sigma} \in \Gamma$.  It is clear from the definition of the action of $\sigma$ on $Q(\I'/\Qq\I')$ that $\tilde{\sigma} \in D(\Pp_0' | \Qq)$ and $\Phi(\tilde{\sigma}) = \sigma$.  
\end{proof}

\begin{remark}
Suppose that $\Qq$ lies over $P_{k, 1}$ with $k$ divisible by $p - 1$.  Then $f_\Pp \in S_k(\Gamma_0(N), \chi)$.  Under assumption \eqref{quadratic Nebentypus} it follows from Lemma \ref{twist characters and Nebentypus} that all twist characters of $f_\Pp$ are quadratic.
\end{remark}

\end{document}